\documentclass[11pt,a4paper]{amsart}

\usepackage{amssymb}
\usepackage{amsfonts}
\usepackage[all]{xy}
\usepackage{mathrsfs}
\usepackage{graphicx}
\usepackage{stmaryrd}
\usepackage{cite}
\usepackage{hyperref}

\newtheorem{thm}{Theorem}[section]
\newtheorem{prop}[thm]{Proposition}
\newtheorem{lem}[thm]{Lemma}
\newtheorem{cor}[thm]{Corollary}

\theoremstyle{definition}
\newtheorem{definition}[thm]{Definition}
\newtheorem{example}[thm]{Example}
\newtheorem{remark}[thm]{Remark}
\newtheorem{notation}[thm]{Notation}
\newtheorem{convention}[thm]{Convention}
\newtheorem{construction}[thm]{Construction}
\numberwithin{equation}{section}

\renewcommand{\d}{\mathfrak{d}} 
\DeclareMathOperator{\GL}{GL}
\DeclareMathOperator{\ord}{ord}
\DeclareMathOperator{\Spec}{Spec}
\DeclareMathOperator{\Proj}{Proj}
\DeclareMathOperator{\Star}{Star}
\DeclareMathOperator{\Cone}{Cone}
\DeclareMathOperator{\lcm}{lcm}

\newcommand{\R}{\mathbf{R}}       
\newcommand{\Z}{\mathbf{Z}}
\newcommand{\Q}{\mathbf{Q}}
\newcommand{\C}{\mathbf{C}} 
\newcommand{\N}{\mathbf{N}}
\newcommand{\cX}{\mathcal{X}} 
\renewcommand{\P}{\mathbf{P}} 
\DeclareMathOperator{\ev}{ev}
\DeclareMathOperator{\Spf}{Spf}
\DeclareMathOperator{\Int}{Int}

\begin{document}

\title[Smoothing Toric Fano Surfaces]{Smoothing Toric Fano Surfaces Using the Gross--Siebert Algorithm}

\author{Thomas Prince}
\address{Department of Mathematics, Imperial College London, 180 Queen's Gate, London SW7 2AZ}
\email{t.prince12@imperial.ac.uk}

\begin{abstract}
A toric del~Pezzo surface $X_P$ with cyclic quotient singularities determines and is determined by a Fano polygon $P$.  We construct an affine manifold with singularities that partially smooths the boundary of $P$; this a tropical version of a $\Q$-Gorenstein partial smoothing of $X_P$.  We implement a mild generalization of the Gross--Siebert reconstruction algorithm -- allowing singularities that are not locally rigid -- and thereby construct (a formal version of) this partial smoothing directly from the affine manifold.  This has implications for mirror symmetry: roughly speaking, it implements half of the expected mirror correspondence between del~Pezzo surfaces with cyclic quotient singularities and Laurent polynomials in two variables.
\end{abstract}

\maketitle
\section{Introduction}
\label{sec:introduction}

There has been much recent interest in the classification of log del~Pezzo surfaces up to deformation -- in particular log del~Pezzo surfaces have been classified in index at most two by Alexeev--Nikulin~\cite{AN06} and in index three by Fujita--Yasutake~\cite{FY14}. Here we analyse $\Q$-Gorenstein deformations of del~Pezzo surfaces with cyclic quotient singularities, exploring a rich combinatorial structure predicted to exist by Mirror Symmetry. The results we obtain are significantly less detailed than the known classification theorems, but they apply in greater generality:  to all log del~Pezzo surfaces with cyclic quotient singularities.

The current work is inspired by a program, laid out in \cite{Fano} by Coates--Corti--Galkin--Golyshev--Kasprzyk, which conjectures the existence of a combinatorial structure on the set of toric varieties to which a given Fano variety degenerates. As mentioned, this conjecture is a manifestation of \emph{Mirror Symmetry}, which conjectures a general correspondence between Fano varieties and certain \emph{Landau--Ginzburg models}. These are still vague conjectures even for Fano manifolds, but for dimension two, specifically for orbifold del~Pezzo surfaces, my coauthors and I made the conjecture precise in~\cite{Overarching}.

Indeed, Theorem~3 in \cite{Overarching} states, roughly speaking, that the collection of toric varieties to which a surface $X$ admits a $\Q$-Gorenstein degeneration is precisely a collection of so-called \emph{mutation classes} of Fano polygons; a central conjecture in that paper is that there is in fact only one mutation class for each surface $X$. Theorem~3 of \cite{Overarching} was inspired and proven via two observations: First, that $X$ is mirror-dual\footnote{In the sense of \cite{Fano}, i.e.~the coincidence of local systems associated to the Picard--Fuchs equations for the Laurent polynomials and the quantum differential equations for $X$.} to a collection of Laurent polynomials, and that (some of) these Laurent polynomials are related by special birational transformations called mutations.  Passing to Newton polygons $P$ and $P'$,  mutation defines a purely combinatorial operation taking $P$ to $P'$. 
Second, by a theorem of Ilten \cite{Ilten}, for any mutation there is a $\Q$-Gorenstein family over $\P^1$ for which $X_P$ and $X_{P'}$ are the fibers over $0$,~$\infty$ respectively. (Note however that $X$ need not be the general fiber of this family.) The result of \cite{Overarching} follows as we see that if $X$ admits a $\Q$-Gorenstein degeneration to $X_P$ it must admit one to $X_{P'}$.

While this reveals an interesting structure we so far lack a geometric understanding of why this structure should appear.  Similarly, while we have a good understanding of the 1-strata (Ilten pencils) of the `$\Q$-Gorenstein parameter space' we lack a combinatorial description of the higher dimensional strata. The current work tackles these two issues by introducing an intermediate object, an \emph{affine manifold with singularities}, which incorporates the combinatorial (mutation) structure of the Fano polygons. Utilizing the techniques and algorithm introduced by Gross--Siebert in \cite{GS1,GSlog} we view this object as associated to a deep (maximal) degeneration of the surface $X$. In fact via the Gross--Siebert reconstruction algorithm we can not only pass from an affine manifold to a (formal version of) $X$ but also construct (a formal version of) the total space of the $\Q$-Gorenstein smoothing of $X_P$, by allowing the functions defining certain log-structures to vary and degenerate. Thus we both:
\begin{itemize}
\item  canonically construct affine manifolds from mutation classes of Fano polygons; and 
\item canonically recover precisely the $\Q$-Gorenstein smoothing families from these affine manifolds.
\end{itemize}

At the end we return to interpret the Ilten pencils, which were our original inspiration, in terms of these affine manifolds. We find that, again going via a more degenerate limit, the Ilten pencil can be reconstructed by applying the Gross--Siebert algorithm to a simple affine manifold (or family of affine manifolds) formed by moving a single singularity in the affine structure.

\begin{remark}
The affine manifolds we consider also play an crucial, if somewhat implicit, role in the recent works of Gross--Hacking--Keel \cite{GHK,GHK2}. In particular, in cases where the anti-canonical system of $X$ contains no smooth elliptic curves, we are able to apply the main result of \cite{GHK}. Indeed, in this case we may associate to an affine manifold with singularities~$B$ a log Calabi--Yau pair $(Y,D)$, following the ideas of \cite{GHK}. The main result of \cite{GHK} produces an algebra of theta functions defining the universal family of the log Calabi--Yau 
\[
\mathring{X} := X \setminus \text{(anticanonical divisor)};
\]
$\mathring{X}$ appears in the current work by the application of the Gross--Siebert algorithm to the interior of $B$. In fact we can say more.  The construction we use provides not only a log Calabi--Yau pair $(Y,D)$ but a toric model: $(Y,D) \rightarrow (\bar{Y},\bar{D})$ which defines a torus chart on $Y \backslash D$. Thus we have constructed a bijection between a collection of $\Q$-Gorenstein toric degenerations of an orbifold del Pezzo surface $X$ (which are related by sliding the singularities or \emph{moving worms}) and a set of torus charts on the mirror-dual variety $Y \backslash D$.  The existence of such a bijection is a guiding heuristic of the programme set out in \cite{Fano} and \cite{Overarching}.
\end{remark}


\subsection*{Overview}

Recall that in toric geometry a polygon $P^\vee$ is the base of a \emph{special Lagrangian torus fibration} given by the moment map for $X_P$.  More generally given any special Lagrangian torus fibration it is well known that the base manifold $B$ carries a canonical \emph{affine structure} \cite{SYZ}.  The \emph{SYZ conjecture} states that mirror symmetry takes the points of a Calabi--Yau manifold to a family of special Lagrangian tori in the mirror variety.  Thus we expect that $X$ should carry a special Lagrangian torus fibration with base\footnote{Here, since $X$ is Fano rather than Calabi--Yau, the affine manifold~$B$ will have non-empty boundary and its Legendre dual $\breve{B}$ will be non-compact.}~$B$, and that the mirror $X^\vee$ should carry a dual special Lagrangian torus fibration over the same base $B$.  The superpotential $W$ in this formulation is a count of holomorphic discs in $X$ of Maslov index~2 \cite{Aur1,Aur2}.  Indeed, 
one should recover the Laurent polynomial superpotentials explicitly by counting the tropical analogues of these discs, which are so-called \emph{broken lines} in the affine manifold $B$~\cite{GPS,TropicalLG,TropGeom}.  We do not construct the special Lagrangian torus fibration on $X^\vee$ in this paper, nor do we analyse counts of broken lines.  (We will do this in future work.)  We concentrate instead on structures on~$X$.

Starting from the polygon $P$, thought of as an affine manifold, we construct `tropical deformations' of $P$ by exchanging corners for \emph{singularities} in the affine structure.  In this way we form a parametrized family $B_t$ of affine manifolds in which these singularities move -- this is referred to as \emph{moving worms} in \cite{KoSo}. Our task then is to build a variety carrying a special Lagrangian torus fibration, which has the affine manifold $B_t$ as its base.  However this is a familiar problem: recent approaches to proving asymptotic versions of the SYZ conjecture by Gross--Siebert~\cite{GS1, TropGeom} and Kontsevich--Soibelman~\cite{KoSo} involve forming the base manifold by \emph{toric degeneration}, taking a \emph{Legendre dual} affine structure equipped with extra data, and exhibiting an algorithm to reconstruct the mirror variety from the base manifold.  This reconstruction algorithm precisely allows us to pass from the base manifold $B_t$ to a formal or analytic neighbourhood of a central fiber of a toric degeneration, although one typically does not recover the existence of a genuine special Lagrangian fibration from these methods.

Using the Gross--Siebert algorithm to reconstruct the desired families requires first attending to a number of details. Firstly one must pass from the affine manifold $B_t$ to the central fiber $X_0(B_t,\mathscr{P},s)$ of a toric degeneration. This central fiber is independent of $t$. We must then define a notion of compatibility between the family $\{B_t\}$ of affine manifolds and a family of \emph{log-structures} on the central fiber. Since this setting is not \emph{locally rigid} in the sense of~\cite{GS1} we shall describe the algorithm in some detail and give the slight amendments required in this context (Sections~\ref{sec:log_structures}--\ref{sec:gluing}).  We then consider (in Section~\ref{sec:local_models}) explicit descriptions of the schemes produced by the Gross--Siebert algorithm near to boundary zero-strata of $B_t$ and show that these are compatible with the $\Q$-Gorenstein smoothings of these singularities.  We combine these results in the following theorem, where we show that the Gross--Siebert algorithm may be applied fiberwise, producing a family of formal families each of which is a thickening of $X_0(B_t,\mathscr{P},s)$.

\begin{thm}\label{thm:families}
Given a Fano polygon $Q$ let $\mathscr{P}$ be the polygonal decomposition via the spanning fan and let $s$ be trivial gluing data.  From this data we may form a flat family $\mathcal{X}_Q \rightarrow \Spec\C[\alpha]\llbracket t \rrbracket$ such that:
\begin{itemize}
\item Fixing a nonzero $\alpha$ the restriction of $\cX$ over $\Spec \C\llbracket t\rrbracket$ is the flat formal family produced by the Gross--Siebert algorithm.
\item Fixing $\alpha = 0$ the restriction of $\cX$ over $\Spec \C\llbracket t\rrbracket$ is precisely the Mumford degeneration of the pair $(Q,\mathscr{P})$.
\item Fixing $t=0$, the restriction of $\cX$ over $\Spec \C[\alpha]$ is $X_0(Q, \mathscr{P}, s) \times \Spec \C[\alpha]$.
\item For each boundary zero-stratum $p$ of $X_0(Q, \mathscr{P}, s)$ there is neighbourhood $U_p$ in $\cX$ isomorphic to a family $\mathcal{Y}\rightarrow \Spec\C[\alpha]$ obtained by first taking a one-parameter $\Q$-Gorenstein smoothing of the singularity of $X_Q$ at $p$, taking a simultaneous maximal degeneration of every fiber and restricting to a formal neighbourhood of the central fiber.
\end{itemize}
\end{thm}

The main difficulty in doing this is that the construction of the formal family at the fiber $\alpha=0$ is different from the other fibers -- indeed, the log-structure there is a section of a different bundle.  We overcome this by giving an explicit description of various rings involved in the Gross--Siebert algorithm.

The use of an order-by-order scattering process means that, outside of certain specific cases, we are unable to write down explicit expressions for the general fibers of the toric degenerations we consider.  One particularly striking case in which this is possible is the case there is only a single (simple) singularity; analysing this case leads us to recover a theorem of Ilten (\!\!\cite{Ilten}):
\begin{thm} \label{thm:Ilten_in_introduction}
For any combinatorial mutation from $P$ to $P'$ there is a flat family $\mathcal{X} \rightarrow \C^2$ such that the fiber over zero is $X_P$ and the fiber of $\infty$ is $X_{P'}$.
\end{thm}
\noindent See \S\ref{sec:ilten}.  We refer to the corresponding family of affine manifolds the \emph{tropical Ilten family}.  The Ilten pencil, which has base $\P^1$, is obtained from the family in Theorem~\ref{thm:Ilten_in_introduction}, which has base $\C^2$, by taking the quotient by radial rescaling.  

\subsection*{Acknowledgements}

This work lies within the Gross--Siebert program.  The author learned a great deal about the Gross--Siebert program from the lectures by Mark Gross at the 2013 MIT-RTG Mirror Symmetry Workshop at Big Bear Lake,~CA.  We thank Mark Gross and the participants of the workshop for many useful explanations.  We thank Tom Coates, Alessio Corti, Alexander Kasprzyk and members of the Fanosearch group for useful and inspiring comments and conversations.  The author is supported by an EPSRC Prize Studentship.

\section{Affine Manifolds With Singularities}
\label{sec:affine_manifolds}

In this section we shall introduce affine manifolds with singularities.  From our point of view these are tropical or combinatorial avatars of algebraic varieties.  We shall briefly discuss the connection to the SYZ conjecture, which also offers a first justification for this point of view: the base of a special Lagrangian torus fibration naturally has the structure of an affine manifold.  By way of example: given a toric variety we can form an affine manifold via its moment map, isomorphic to a polygon $Q$.  We shall then consider a suitable notion of families of these objects and specifically how one can `smooth' the corners of a polygon by replacing them with singularities in the interior.   In particular, starting with a Fano polygon $Q$ this will form a combinatorial analogue of the $\Q$-Gorenstein deformations of the associated del~Pezzo surface: indeed, the bulk of the later sections is devoted to reconstructing such an algebraic deformation from this combinatorial data.

\begin{definition} 
  \label{def:affine_manifold_with_singularities}
An \emph{affine manifold with singularities} is a piecewise linear (PL) manifold $B$ together with a dense open set $B_0 \subset B$ and a maximal atlas on $B_0$ that is compatible with the topological manifold structure on $B$ and which makes $B_0$ a manifold with transition functions in $\GL_n\left(\R\right) \rtimes \R^n$.  
\end{definition}

\begin{remark}
  To give a maximal atlas on $B_0$ with transition functions in $\GL_n\left(\R\right) \rtimes \R^n$ is the same as to give the structure of a smooth manifold on $B_0$ together with a flat, torsion-free connection on $TB_0$.
\end{remark}

Following Kontsevich--Soibelman~\cite{KoSo} we can reinterpret this definition in terms of the sheaf of affine functions:
\begin{definition}
The \emph{sheaf of affine functions} $\text{Aff}_{\Z,X}$ on an affine manifold $X$ is the sheaf of functions which, on restriction to any affine chart, give affine functions.
\end{definition}

\begin{lem}[\cite{KoSo}]
Given a Hausdorff topological space $X$, an affine structure on $X$ is uniquely determined by a subsheaf $\text{Aff}_{\Z,X}$ of the sheaf of continuous functions on $X$, such that locally $(X, \text{Aff}_{\Z,X})$ is isomorphic to $(\R^n,\text{Aff}_{\Z,\R^n})$.
\end{lem}

\begin{remark}
$\text{Aff}_{\Z,X}$ is a sheaf of $\R$-vector spaces, but as the product of two affine functions is not in general affine, it is not a sheaf of rings.  There is a subspace analogous to the maximal ideal of a local ring, given by the kernel of the evaulation map $\ev \colon {\text{Aff}_{\Z,B}}_p \rightarrow \R$.
\end{remark}

\begin{definition}
  A morphism of affine manifolds is a continuous map $f\colon B \rightarrow B'$ that is compatible with the affine structures on $B$ and $B'$.
\end{definition}

\begin{definition}
  \label{def:tropical_and_integral}
  If the transition functions for $B_0$ lie in $\GL_n\left(\Z\right) \rtimes \R^n$, we say that the affine manifold is \emph{tropical}; this is equivalent to insisting that there is a covariant lattice in $TB_0$ preserved by the connection.   If the transition functions lie in $\GL_n\left(\Z\right) \rtimes \Z^n$ then the affine manifold is called \emph{integral}; this is equivalent to insisting that there  there is a lattice in $B_0$ preserved by the transition functions.
\end{definition}

\begin{notation}
  We shall always assume that affine manifolds are tropical, so there is a lattice $\Lambda_x \subseteq T_xB_0$.  We set $\Delta := B \setminus B_0$, and refer to it as the \emph{singular locus} of the affine structure.  If $\Delta = \varnothing$ then the corresponding affine manifold is called \emph{smooth}.
\end{notation}

The relevance of affine manifolds to mirror symmetry comes from the SYZ conjecture~\cite{SYZ}, which roughly speaking states that a pair of mirror manifolds should carry special Lagrangian torus fibrations that are dual to each other.  If one is in such a favourable setting, the base of this fibration carries a pair of (smooth) affine structures, and, in this so-called \emph{semi-flat} setting, one can reconstruct the original pair of manifolds, $X$,~$\breve{X}$ from the affine structures.  Indeed from a given smooth tropical affine manifold $B$ one may construct a pair of manifolds $X = TB/\Lambda$, $\breve{X} = T^\ast B/\breve{\Lambda}$ where $\Lambda$ is the covariant lattice in $TB$ defined by the affine structure and $\breve{\Lambda} \subset T^* B$ is the dual lattice.  The manifold $X$ carries a canonical complex structure and the manifold $\breve{X}$ carries a canonical symplectic structure~\cite{Gross:SYZ}.  To endow $X$ with a symplectic structure, respectively $\breve{X}$ with a complex structure, we need to attach to $B$ a (multivalued, strictly) convex function $\varphi \colon B \to \R$.  Here there is a canonical choice for $\varphi$: the K\"ahler potential for the McLean metric on $B$~\cite{Gross:SYZ,McLean}.  The convex function $\varphi$ allows us to define the \emph{Legendre dual} $\breve{B}$ of the affine manifold $B$, and one can show that Legendre duality $B \leftrightarrow \breve{B}$ interchanges the pair of affine structures coming from a special Lagrangian torus fibration.  This identification of $TB/\Lambda$ with $T^\ast \breve{B}/\breve{\Lambda}$, and $T^\ast B/\breve{\Lambda}$ with $T \breve{B}/\Lambda$ recovers, as promised, the mirror pair of K\"ahler manifolds $X$,~$\breve{X}$.

\begin{example}
  The standard examples of affine manifolds without boundary or singularities are tori, which have natural flat co-ordinates.  For example, taking the base manifold $B$ to be $S^1$ and endowing $X = T B/\Lambda$ with the canonical complex structure described above yields an elliptic curve $X$.
\end{example}

\begin{example}
  \label{ex:toric}
  Consider a polytope $P \subset \R^n$.  The inclusion $P \to \R^n$ equips the interior $B$ of $P$ with the structure of an affine manifold.  The non-compact symplectic manifold $T^\star B/\breve{\Lambda}$ admits a Hamiltonian action of $(S^1)^n$ for which the moment map is given by the projection to $B$.  It is clear in such examples how to extend the construction of this torus bundle over $B$ to the boundary strata of $P$: indeed this is nothing other than Delzant's construction of symplectic toric varieties from their moment polytopes~\cite{Delzant}.
\end{example}

\begin{remark}
  \label{rem:Auroux}
  As the last example demonstates we shall often be interested in cases where $B$ (or $B_0$) is a \emph{manifold with corners}.  A discussion of mirror symmetry for toric varieties from this perspective may be found in~\cite{Aur1}.  Auroux explains there that
  one may define complex co-ordinates on $\breve{X}$ by taking the areas of certain holomorphic cylinders in $X$, together with certain $U(1)$-holonomies.  After adding compactifying divisors to $X$, these cylinders become discs, and so co-ordinates on the mirror manifold $\breve{X}$ are determined by computing the areas of certain holomorphic discs.  In the toric setting (Example~\ref{ex:toric}) this construction gives global co-ordinates on $\breve{X}$.  In general, and certainly in our case (where singularities are present), computing areas of holomorphic discs will give only local co-ordinates on $\breve{X}$, with the transition functions between these co-ordinate patches reflecting \emph{instanton corrections}.  From this perspective, much of the rest of this article consists of a careful analysis of the instanton corrections in our setting: computing them explicitly where possible, and determining how they vary in certain simple families.  We return to this point in the Conclusion.
\end{remark}

\subsection{Focus-focus singularity - the local model}

In the rest of this paper, we will primarily be concerned with affine manifolds that arise from polytopes, but rather than taking the polytope $Q$ itself as the affine manifold, we shall instead \emph{smooth the boundary}, exchanging the corners of $Q$ for singularities in the interior of the polytope.  The local model for this situation is as follows.  Consider a two-dimensional affine manifold $S_\kappa$, where $\kappa$ is a parameter, defined via a covering by two charts:
\begin{align*}
  U_1 = \R^2\setminus\big(\R_{\geq 0}\times\{0\}\big) &&
  U_2 = \R^2\setminus\big(\R_{\leq 0}\times\{0\}\big)
\end{align*}
with transition function $\phi$ from $U_1$ to $U_2$ given by:
\[
(x,y) \mapsto
\begin{cases}
  \left(x,y\right) & y > 0\\
  \left(x + \kappa y,y\right) & y < 0
\end{cases}
\]
The transition function is piecewise-linear: on the upper half-plane it is the identity transformation, and on the lower half-plane it is a horizontal shear with parameter~$\kappa$.  We will assume throughout that $\kappa \in \Z$; in this case, the affine manifold $S_\kappa$ is integral.  We will consider only affine manifolds with singularities that are locally modelled on some $S_\kappa$.

\begin{definition}
  A \emph{singularity of type $\kappa$} in an affine manifold $B$ is a point $p \in \Delta$ such that $p \not \in \partial B$ and that there is a neighbourhood of $p$ isomorphic as an affine manifold to a neighbourhood of $0 \in S_\kappa$.
\end{definition}

\begin{convention}
  Henceforth any affine manifold $B$ that we consider will be two-dimensional and such that each $p \in \Delta$ is a singularity of type $\kappa_p$ for some $\kappa_p \in \Z$.  In particular, the singular locus $\Delta$ of $B$ is disjoint from the boundary of $B$.
\end{convention}

We will be primarily interested in one-parameter families of such affine structures, and in applying the Gross--Siebert algorithm `fiberwise' to reconstruct a degenerating family.

\begin{remark}\label{rem:fibs}
  The Gross--Siebert algorithm for surfaces cannot be applied to certain `illegal' configurations: one needs to insist that both monodromy-invariant lines and the rays introduced by scattering miss the singular locus.  In practice one often guarantees this by ensuring that singularities have irrational co-ordinates.  (In this context, monodromy-invariant lines and rays have rational slope.)  But this approach generally precludes moving the singularities.  As we shall see, smoothing the corners of a polygon is a particularly fortunate setting, where one can freely slide singularities along monodromy-invariant lines without risking illegal configurations.
\end{remark}

\subsection{Corner smoothing - the local model}
\label{sec:local_model}

We shall now construct a local model for a degeneration. The most general definition of `family of affine manifolds' we shall need consists of locally trivial families of affine structures together with finitely many copies of this local model. 

Fix a rational, convex cone $C$ in $\R^2$ and denote the primitive integral generators of its rays by $v_1$ and $v_2$.   Fix a rational ray $L$ contained in the interior of $C$, let $\ell$ be the primitive integer generator of $L$, and fix an integer $k$ such that the rational cone generated by $v_1$ and $v_2 - k \ell$ either contains $L$ or is itself a line in $\R^2$.  We shall construct a topological manifold $\mathcal{B}_{C,L,k}$, together with a sheaf of affine functions on $\mathcal{B}_{C,L,k}$ and a map $\pi_k\colon \mathcal{B}_{C.L.k} \rightarrow \R_{\geq 0}$ of affine manifolds (where $\R_{\geq 0}$ has its canonical affine structure).

\begin{definition} \label{def:local_model}
As a topological manifold $\mathcal{B}_{C,L,k}$ is equal to $C \times \R_{\geq 0}$.  We give it an affine structure via an atlas with $k+1$ charts.  We define each chart $U_i = \left(C \times \R_{\geq 0}\right)\backslash V_i$ where each $V_i$ is a subset of $L \times \R_{\geq 0}$ as follows:
\begin{align*}
  & V_0 =  \{(x\ell,t) : 0 \leq x \leq t < \infty \} \\
  & V_i = \{(x\ell,t) : \text{$0 \leq x \leq i t$ or $(i+1)t \leq x < \infty$} \} && \text{for $0 < i < k$} \\
  & V_k =  \{(x\ell,t) : 0 < t \leq x < \infty \} 
\end{align*}
with transition functions fixed by the requirement that, for $1 \leq i \leq k$ and for all $t>0$, the charts $U_{i-1}$,~$U_i$ make the point $(i t \ell,t)$ in the fiber $C \times \{t\}$.
\end{definition}

Note that this only makes the complement of $(0,0)$ an affine manifold, as the origin is in the closure of the singular locus but not contained in it.  Despite this, the sheaf of affine functions is still defined in a neighbourhood of $(0,0)$.

\begin{remark}
  Later on we will restrict this family to a subset, replacing $C \times \R_{\geq 0}$ by $U \times [0,T)$ where $U$ is a neighbourhood of the origin in $C$ and $T$ is sufficiently small that there are $k$ singular points on the fiber $U \times \{T\}$.
\end{remark}

\begin{remark}
  There is an obvious generalization of this local model, which would allow the construction of more complicated degenerations.  Rather than introduce a singularity of type~$1$ for each $1 \leq i \leq k$, we may consider a partition $\mathbf{k} = (k_1, \cdots, k_m)$ of $k$ and construct a version $\mathcal{B}_{C,L,\mathbf{k}}$ of $\mathcal{B}_{C,L,k}$, in which the fiber over $t \in \R_{\geq 0}$ contains a singularity of type $k_i$ at $(it\ell,t)$ for $1 \leq i \leq m$.
\end{remark}

\subsection{One-parameter families}

\begin{definition}

We define a \emph{one-parameter degeneration of affine structures} to be a topological manifold with corners $\mathcal{B}$ and a continuous map:
\[
\pi \colon \mathcal{B} \rightarrow \R_{\geq 0}
\]
such that:
\begin{itemize}
\item for some finite set $S$ of points in the boundary of $\pi^{-1}(0)$, $\mathcal{B} \backslash S$ is an affine manifold and $\pi$ is a locally trivial map of affine manifolds; and
\item  for each $p \in S$ there is a neighbourhood $U$ of $p$ in $\mathcal{B}$ and a triple $(C,L,k)$ such that $U$ is isomorphic, as an affine manifold, to an open set of $\mathcal{B}_{C,L,k}$, via an isomorphism that identifies $\pi$ with $\pi_{C,L,k}$.
\end{itemize}
\end{definition}

\begin{remark}
  We will need to consider only one-parameter degenerations of affine structure such that a neighbourhood of the central fiber is locally modelled on $\mathcal{B}_{C,L,k}$ for various triples $(C,L,k)$, possibly with $k=0$.  
\end{remark}

\noindent It would be interesting to consider the generalization of this notion to families over arbitrary affine manifolds, and the associated moduli problems.  

\subsection{Polygons and Singularity Content}

In this section we shall construct a one-parameter degeneration of affine structures from a given Fano polygon which partially smooths each vertex, in the sense we have described above.  This is closely related to the notions of \emph{singularity content}, \emph{class~T} and \emph{class~R} singularities which appear in \cite{SingCon}.

A polygon $P$ is \emph{Fano} if it is integral, contains the origin and has primitive vertices. Fix such a polygon $P$ and denote its polar polygon $Q := P^\circ$.  In particular the origin is contained in the interior of $Q$. Fix a polyhedral decomposition $\mathscr{P}$ of $Q$ by taking the spanning fan and restricting this fan to the polytope $Q$.

Fix a vertex $v \in \text{Vert}(Q)$. The decomposition $\mathscr{P}$ induces a canonical choice of 1-cell $L_v$ for each $v \in \text{Vert}\left(Q\right)$: the 1-cell which is a cone of the spanning fan of $Q$.  Consider the subset $U_v = \Star(v) \subset Q$; $U_v$ is isomorphic to an open subset of a cone $C_v$ with origin $v$ and bounded by the rays containing each edge of $Q$ incident to $v$.  The 1-cell $L_v$ becomes the restriction of a ray in this cone.  To form a triple $(C_v,L_v,k)$ as in \S\ref{sec:local_model}
we still require the choice of a suitable integer $k$.

\begin{definition}\label{def:sing_content}
We shall refer to the maximal integer $k$ such that $(C,L,k)$ satisfy the conditions just before Definition~\ref{def:local_model} as the \emph{singularity content} of the pair $(C,L)$.
\end{definition}

For each vertex $v \in Q$ denote by $k_v$ the singularity content of $(C_v,L_v)$, and choose a function $k\colon \text{Vert}(Q) \rightarrow \Z_{\geq 0}$ such that $0 \leq k(v) \leq k_v$.  We may now form the families $\mathcal{B}_{C_v,L_v,k(v)}$.  Restrict each family to $U_v \times [0,T_v)$ where the fiber over $T_v$ contains $k(v)$ singular points.

\begin{definition}
Let $\pi_{Q,k}\colon \mathcal{B}_{Q,k} \rightarrow [0,T)$ where $T = \min_v(T_v)$ be the following one-parameter degeneration of affine manifolds.  As a topological manifold it is $Q \times [0,T)$, covered by the following charts:
\begin{itemize}
\item $U_v \times [0,T)$ as defined above for each vertex of $Q$ and,
\item $W \times [0,T)$ where $W$ is a neighbourhood of the origin.
\end{itemize}
We may regard $U_v \times [0,T)$ as an affine manifold, with affine structure induced from $\mathcal{B}_{C_v,L_v,k(v)}$.  We  define the affine structure on $\mathcal{B}_{Q,k}$ by insisting that the transition functions between the $k(v)^{\text{th}}$ chart of $U_v \times [0,T)$ and the $k(v')^{\text{th}}$ chart of $U_{v'}\times [0,T)$ is the identity for vertices $v$ and $v'$, and the transition function between each of these charts and $W \times [0,T)$ is also the identity.
\end{definition}

\begin{notation}
We will typically wish to smooth the corners as much as possible, so we use the notation
$\pi_Q\colon \mathcal{B}_Q \rightarrow \R_{\geq 0}$ for the map $\pi_{Q,k}\colon \mathcal{B}_{Q,k} \rightarrow \R_{\geq 0}$ where $k$ is the function sending each vertex to its singularity content.
\end{notation}

We next show that our notion of singularity content (Definition~\ref{def:sing_content}) coincides with that of Akhtar--Kasprzyk~\cite{SingCon}. We recall that given a Fano polygon $P \subset N_{\R}$ we may consider an edge $e$ containing $v_1,v_2 \in \text{Vert}(P)$.  The edge defines an (inward-pointing, primitive) element of the dual lattice $w \in M$ such that $w(e)$ is a constant non-zero integer $l$.  We may also consider the cone over the edge $e$, which we denote $C_e$.  Let $\theta$ denote the lattice length of the line segment from $v_1$ to $v_2$.  Writing $\theta = n l + r$ where $0 \leq r < l$, decomposes $C_e$ into:
\begin{enumerate}
\item A collection of $n$ cones whose intersection with the affine hyperplane defined by $w(v) = l$ is a line segment of length $l$; and, if $r>0$,
\item A single cone of width $r<l$.  This is the \emph{residual} cone from~\cite{SingCon}.
\end{enumerate}
If $C_e$ contains no residual cone then we say that $C_e$ is of \emph{class~T}.  Akhtar--Kasprzyk call $n$ the singularity content of $C_e$.

Consider an edge $e$ of $P$ with vertices $v_1, v_2$; this determines a vertex $v_e$ of the polar polygon $Q$, and thus a cone $C$ with origin at $v_e$, having rays dual to $v_1$ and $v_2$.  
The normal direction to $e$ defines a ray in $Q$ passing though $v_e$ and the origin.  Thus to the polygon $P$ and edge $e$, we may associate a pair $(C,L)$.

\begin{lem}
The singularity content of $(C,L)$ as in Definition~\ref{def:sing_content} is equal to the singularity content of the cone over the edge $e$ as defined in~\cite{SingCon}.
\end{lem}
\begin{proof}

After a change of co-ordinates in $N$ we may assume that the vertices $v_1$,~$v_2$ of $e$ are $(a_1,-h)$ and $(a_2,-h)$ respectively.  The rational polygon $Q$ then has a vertex $v_e = (0,-1/h)$ and edges which contain this vertex in directions $(-h,-a_1)$ and $(h,a_2)$.  This defines the cone $C$ above.  The ray L is vertical, and the singularity content of $(C,L)$ is:
\[
\max\{k \in \Z_{\geq 0} : a_2 - kh \geq a_1 \}
\]
This is the largest $k$ such that $kh \leq a_2 - a_1$, and since $\theta = a_2 - a_1$ is the lattice length of the edge $e$, we see that the two definitions of singularity content coincide.
\end{proof}

\begin{definition}
Let $B$ be an affine manifold with singularities and corners, and $\mathscr{P}$ a polygonal decomposition of $B$.  This pair is of \emph{polygon type} if it is isomorphic to a fibre of a family $\pi_{Q,k} \colon \mathcal{B}_{Q,k} \rightarrow \R_{\geq 0}$.
\end{definition}

\section{From Affine Manifolds to Deformations: an Outline}

We are now nearly in a position to apply the Gross--Siebert reconstruction algorithm to our base manifolds.  Since we will require a slight generalization of the Gross--Siebert algorithm and since some of the details will be important later in the paper, we present the procedure in some detail.  As a consequence sections ~\ref{sec:log_structures} to ~\ref{sec:gluing} draw heavily on the paper~\cite{GS1} of Gross--Siebert and the book~\cite{TropGeom} by Gross.

As input data for this algorithm we require a two-dimensional affine manifold with singularities, plus some extra data attached to it.  In section~\ref{sec:log_structures} we describe this extra data, introducing the notion of \emph{log~structure} and \emph{open gluing data}, and explain how these data together determine the central fiber $X_0(B,\mathscr{P},s)$ of a toric degeneration.

In section~\ref{sec:structures} we define the structure on the affine manifold with singularities plus log data, referred to simply as a \emph{structure}, which encodes an $n$th-order deformation of $X_0(B,\mathscr{P},s)$.  Section~\ref{sec:scattering} is then devoted to a description of the process (``scattering'') by which an $n$-structure can be transformed into an $(n+1)$-structure; in other words, an $n$th-order deformation can be prolonged to an $(n+1)$st-order deformation.  Finally we describe in section~\ref{sec:gluing} how to pass from a structure to an $n$th-order deformation of the central fiber.

The rest of the article then applies this reconstruction algorithm to our original problem of smoothing cyclic quotient surface singularities.  This accomplished in a series of steps:

\begin{enumerate}
\item In section~\ref{sec:local_models} we compute explicitly the local model at each boundary zero stratum.
\item In section~\ref{sec:smoothing} we return to the original problem: taking a polygon we show how the tropical family constructed in section~2 may be lifted order by order to give an algebraic family over $\Spf \C[[t]]$.  Away from the central fiber, this is an application of the generalized Gross--Siebert algorithm; near the central fiber, this makes use of the local models computed in section~\ref{sec:local_models}.  We further show that the local models at the vertices are compatible with the canonical cover construction, and thus that the family that we construct is $\Q$-Gorenstein.
\item In section~\ref{sec:ilten} we consider the special case in which a single singularity slides along its monodromy-invariant line from one corner into the opposite edge.  Since there is no scattering diagram to consider, the tropical family here may be lifted to an algebraic family over $\P^1$; once again this algebraic family is $\Q$-Gorenstein.
\end{enumerate}

\section{Log Structures on the Central Fiber}
\label{sec:log_structures}

In secion~\ref{sec:affine_manifolds} we have considered the tropical analogue of smoothing the class-T singularities of a Fano toric surface.  As explained, a version of the Gross--Siebert algorithm will allow us to reconstruct from this an algebraic family, the central fiber of which is itself the restriction to a formal neighbourhood of the central fiber of a degeneration of the Fano toric surface.  The general fiber will be a different formal family with the same central fiber.  The data appended to this central fiber that dicates which smoothing we take is a \emph{log structure}.  In this section we give a very functional description these log structures.  However for a complete explanation of this notion, and its relevance to the Gross--Siebert algorithm, the reader is referred to \cite{GSlog,GS1}.  For the rest of this section we fix a triple $\left(B,\mathscr{P},s\right)$, where $\mathscr{P}$ is a polyhedral subdivison of $B$ into convex, rational polyhedra.  Here $s$ is a choice of \emph{open gluing data}, a concept we will also summarise in this section.

\subsection{Construction of the central fiber}

The method for constructing a scheme from the pair $\left(B,\mathscr{P}\right)$ is straightforward.  Each polygon in the decomposition $\mathscr{P}$ defines a toric variety via its normal fan, and the central fiber is constructed by gluing these along the strata they meet along in $\mathscr{P}$.  Formally speaking, in order to define this gluing, we define a small category associated to a polyhedral decomposition:
\begin{definition}
Let $\mathscr{P}$ also denote the category which has:
\begin{description}
\item[Objects] The strata of the decompostion.
\item[Morphisms] At most a single morphism between any two objects, where $e\colon\omega \rightarrow \tau$ exists iff $\omega \subseteq \tau$.
\end{description}
\end{definition}

We next define a contravariant functor $V \colon \mathscr{P} \Rightarrow \textrm{AffSchemes}$.  Its action on objects is as follows.
Fix a vertex $v \in \mathscr{P}^0$.  At $v$ there is a fan $\Sigma_v \subseteq T_vB$ given by all the strata of $\mathscr{P}$ that meet $v$.  Define $K_\omega$ to be the cone in $\Sigma_v$ defined by the element $\omega \in \mathscr{P}$.

\begin{definition}[of $V$ on zero-dimensional objects]\label{def:V_for_vertices}
The co-ordinate ring of $V\left(v\right)$ is given by the Stanley--Reisner ring of the fan $\Sigma_v$: for lattice points $m_1$,~$m_2 \in |\Sigma_v|$, we set
\[
m_1.m_2 = 
\begin{cases}
  m_1.m_2 & \text{if $m_1$,~$m_2 \in K_\omega$ for some $\omega \in \Sigma_v$}\\
  0 & \text{otherwise}
\end{cases}
\]
\end{definition}

Given a stratum $\tau \in \mathscr{P}$ and a vertex $v$ of $\tau$, we define a fan around $v$:
\[
\tau^{-1}\Sigma_v = \left\{ K_e + \Lambda_{\tau,\R} : \text{$K_e \in \Sigma_v$, $e\colon v \rightarrow \sigma$ factoring though $\tau$} \right\}
\]
recalling from \cite{GS1} that $\Lambda_{\tau, \R}$ is the linear subspace generated by $\tau$ in $T_vB$.  We remark, as in \cite{GS1}, that this subspace depends only on $\tau$ and not on the choice of vertex $v$.
We can now define the image of a stratum $\tau$ under $V$:

\begin{definition}[of $V$ on positive-dimensional objects]
\[ 
V\left(\tau\right) = \Spec k\left[\tau^{-1}\Sigma_v\right] 
\]
where this $k$-algebra is interpreted as the Stanley--Reisner ring, as in Definition~\ref{def:V_for_vertices}.
\end{definition}

We now wish to define the functor $V$ on morphisms.  There is an obvious choice, namely sending a morphism $\tau \to \omega$ to the natural inclusion map $V\left(\tau\right) \rightarrow V\left(\omega\right)$ given by the fan.  However one is free to compose this inclusion map with any choice of toric automorphism of $V\left(\tau\right)$.  The choices of such automorphisms for every inclusion $\omega \hookrightarrow \tau$ form exactly the \emph{Open gluing data} of \cite{GS1}, which we denote by $s$.  This choice is not arbitrary, since $V$ should be functorial: this constraint leads to the precise definition of open gluing data which we shall describe below.  Once the definition of open gluing data is in place, and thus we have a well-defined functor $V$, we may then define the central fiber as the colimit:
\begin{equation}
  \label{eq:colimit}
  \prod_{\omega \in \mathscr{P}}{V\left(\omega\right)} \rightarrow X_0\left(B,\mathscr{P},s\right) 
\end{equation}

\subsubsection{Open Gluing Data:}

In \cite{GS1} the authors explain that the toric automorphisms of an affine piece $V(\tau) = \text{Spec}\left(k\left[\tau^{-1}\Sigma_v\right]\right)$ for $v$ a vertex of $\tau$ are in bijection with elements of a set $\textrm{PM}\left(\tau\right)$ defined as follows.

\begin{definition}\label{def:PM_functions}
Given $\tau \in \mathscr{P}$ and a vertex $v \in \tau$ we define $\textrm{PM}\left(\tau\right)$ to be the set of maps
$\mu \colon \Lambda_v\cap |\tau^{-1}\Sigma_v| \rightarrow k^{\times}$ such that:
\begin{itemize}
\item for any maximal cone $\sigma$ of $\tau^{-1}\Sigma_v$, the restriction of $\mu$ to $\Lambda_v \cap \sigma$ is a homomorphism; and
\item for any two maximal dimensional cones $\sigma$,~$\sigma'$, we have
  \[
  \mu_{\sigma}|_{\Lambda_v \cap \sigma \cap \sigma'} = \mu_{\sigma'}|_{\Lambda_v \cap \sigma \cap \sigma'}.
  \]
\end{itemize}
\end{definition}

\noindent As remarked in \cite{GS1}, whilst this description of $\textrm{PM}\left(\tau\right)$ depends on $v \in \tau$, the set itself is independent of $v$.

\begin{remark} \label{rem:homomorphisms}
An elementary observation we shall use repeatedly in what follows is that the set of homomorphisms $\Lambda_v \cap \sigma \to k^\times$, where $\sigma$ is a maximal dimensional cone, does not depend on the choice of maximal cone $\sigma$.
\end{remark}

\begin{definition} \label{def:open_gluing_data} 
A collection of \emph{open gluing data} is a set
\[
s = \left\{s_e  \in \textrm{PM}\left(\tau\right) \;\mid\; e\colon\omega \rightarrow \tau \right\}
\]
such that if $e\colon \omega \rightarrow \tau$, $f\colon \tau \rightarrow \sigma$ then
$s_f.s_e = s_{f\circ e}$ on the maximal cells where these are defined.  We also insist that $s_{id} = 1$. 
\end{definition}

\noindent The conditions in Definition~\ref{def:open_gluing_data} are precisely those required to ensure that $V$ is a functor.  

\begin{definition}
Collections of open gluing data $s_e$, $s_e'$ are \emph{cohomologous} if there is a collection
$\left\{t_\omega \in \textrm{PM}\left(\omega\right) : \omega \in \mathscr{P} \right\}$ such that\footnote{Here we use the fact that $t_\omega \in \textrm{PM}\left(\omega\right)$ determines a unique element in $\textrm{PM}\left(\tau\right)$, which we also denote by $t_\omega$.}
$s'_e = t_\tau t^{-1}_\omega s_e$ whenever $e \colon \omega \to \tau$.
\end{definition}

\begin{remark} \label{rem:cohomologous}
In \cite{GS1} it is proved that the schemes one obtains via \eqref{eq:colimit} using cohomologous gluing data are isomorphic. 
\end{remark}

\begin{prop}
\label{prop:no_glue}
Let $\left(B,\mathscr{P}\right)$ be of polygon type.  Then all choices of open gluing data are cohomologous.
\end{prop}
\begin{proof}
Fix a polygon $Q$ and label the various strata of $\mathscr{P}$:
\begin{center}
\includegraphics*[viewport=164 405 382 659]{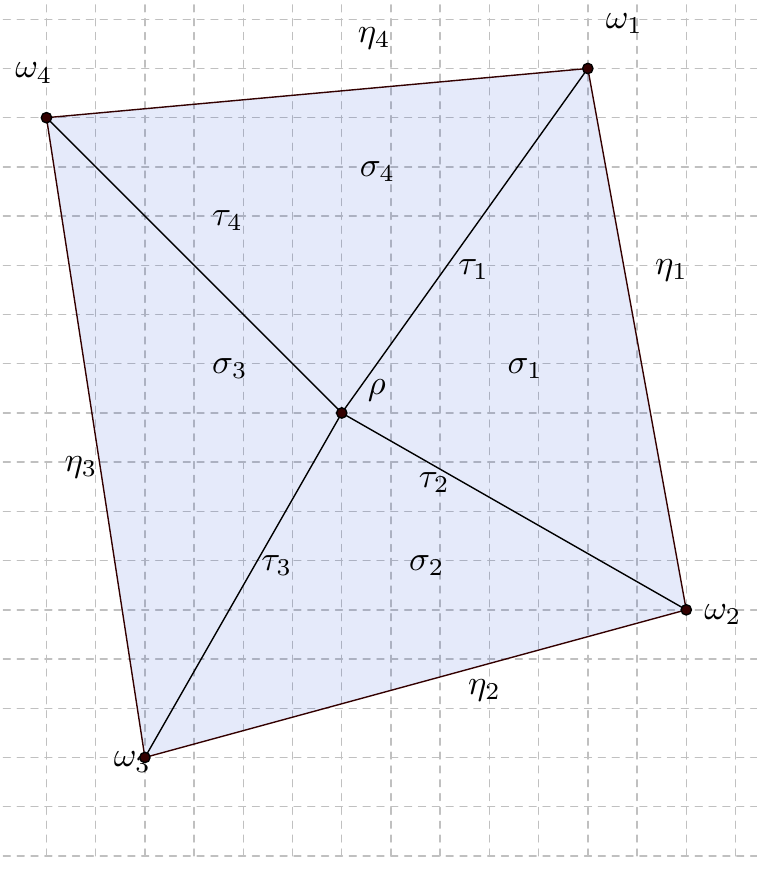}
\end{center}
We need to show that, given any open gluing data $s$ for $\left(B_Q,\mathscr{P}\right)$, we can find a set $\left\{t_\omega \in \textrm{PM}\left(\omega\right) : \omega \in \mathscr{P}\right\}$ such that $s_e = t_\tau t^{-1}_\omega$ for every $e \colon \omega \rightarrow \tau$.  By Remark~\ref{rem:homomorphisms} we have that $\textrm{PM}\left(\eta_j\right) \cong PM\left(\sigma_j\right)$ and $\textrm{PM}\left(\omega_i\right) \cong PM\left(\tau_i\right)$ for all~$i$ and~$j$. Open gluing data $s$ are specified by the following five families of piecewise-multiplicative functions:
\begin{enumerate}
\item $e^1_i\colon \rho \rightarrow \tau_i$
\item $e^2_i\colon \tau_i \rightarrow \sigma_i$, $e^2_i{}'\colon \tau_i \rightarrow \sigma_{i-1}$
\item $e^3_i\colon \omega_i \rightarrow \tau_i$
\item $e^4_i\colon \omega \rightarrow \eta_i$, $e^4_i{}'\colon \omega \rightarrow \eta_{i-1}$ 
\item $e^5_i\colon \eta_i \rightarrow \sigma_i$
\end{enumerate}

We first define open gluing data $s^1$ cohomologous to $s$ by setting $t_\tau = s^{-1}_{e^1_i}$.  Thus $s^1_{e^1_i} = 1$.
Next we observe that $s^1_{e^2_i} = s^1_{e'^2_{i+1}}$, since we have insisted that $s^1_{e^1_i}s^1_{e^2_i} = s^1_{e^1_{i+1}}s^1_{e'^2_{i+1}}$.  Therefore we may define open gluing data $s^2$ cohomologous to $s^1$ by setting $t_{\sigma_i} = (s^1_{e^2_i})^{-1}$.  By construction, $s^2$ associates the trivial element of $\textrm{PM}$ to any morphism between any of $\rho$,~$\tau_i$ and $\sigma_j$.  We now define open gluing data $s^3$ cohomologous to $s^2$ using
$t_{\omega_i} = (s^2_{e^3_i})^{-1}$ and $t_{\eta_i} = (s^2_{e^4_i})^{-1}$.

We claim that the open gluing data $s^3$ are trivial.  First we check $s^3_{e^5_i}$.
We have:
 \[
s^3_{e^5_i} = s^3_{e^4_i}s^3_{e^5_i} = s^3_{e^3_i}s^3_{e^2_i} = 1
\]
where the first equality is the statement that $s^3_{e^4_i} = 1$ together with Remark~\ref{rem:homomorphisms}.  Finally we need to check that $s^3_{e'^4_i} = 1$.  But $s^3_{e'^4_{i+1}}.s^3_{e^5_i} = s^3_{e^4_i}.s^3_{e^5_i}$, so this follows.  Thus any open gluing data for $(B,\mathscr{P})$ are cohomologous to the trivial gluing data.
\end{proof}

Proposition~\ref{prop:no_glue} and Remark~\ref{rem:cohomologous} together show that the scheme obtained from $V$ by gluing (as in equation~\ref{eq:colimit}) is independent of the choice of open gluing data.  Thus we will suppress the dependence on this choice in what follows, assuming that $V$ is constructed using trivial gluing data.

\subsection{A Description Of The Log Structure}

In this section we describe, following~\cite{GS1}, how one may attach a space of log structures to a triple $\left(B, \mathscr{P}, s\right)$.  We begin by describing a sheaf, of which log structures will be (certain) sections.

\begin{definition}\label{def:pre_log}
  Let $\rho \in \mathscr{P}$ be a $1$-cell and let $V_\rho$ be the associated toric variety.  Let $k$ be the total number of singularities of the affine structure on $\rho$, counted with multiplicity\footnote{This is the lattice length of what Gross--Siebert call the \emph{monodromy polytope}, which here is a line segment.}.  Let $v_1$,~$v_2$ be the vertices of $\rho$, and cover $V_\rho$ with two charts $U_i = V\left(v_i\right)\cap V_\rho$.  We shall define a sheaf $\mathcal{N}_\rho$ on $V_\rho$ by setting $\mathcal{N}_\rho\left(U_i\right) = \mathcal{O}_{V_\rho}|_{U_i}$ and using the change of vertex formula
\[
f_{\rho,v_1} = z^{km^{\rho}_{v_1, v_2}}f_{\rho,v_2}
\]
where $m^{\rho}_{v_1, v_2}$ is the primitive vector along $\rho$ from $v_1$ to $v_2$.
\end{definition}

This defines an invertible sheaf.  If the vertices of $\rho$ are integral then $V_\rho$ is canonically isomorphic to $\P^1$ and the sheaf $\mathcal{N}_\rho$ is the line bundle $\mathcal{O}_{\P^1}\left(k\right)$.  In particular the number of zeroes of a generic section of $N_\rho$ is equal to the number of singular points of the affine manifold supported on this stratum, counted with multiplicity.  When the vertices $v_i$ are not integral the 1-strata are canonically identified with  the weighted projective line $\P(a,b)$, where~$a$ and~$b$ are the indices of the respective vertices, and the sheaf $\mathcal{N}_\rho$ is the line bundle $\mathcal{O}\left(k\lcm(a,b)\right)$.  

\begin{remark}
The orbifold structure here depends on the polarization of the central fiber.  In any given example, one can repolarize the central fiber by scaling all the polygons until every vertex is integral; this induces a Veronese embedding on the $1$-strata $\P(a,b)$ considered above.  However this rescaling increases the number of interior integral points we need to consider, and in general leads to much more complicated embeddings.
\end{remark}

\begin{definition}
The sheaf of pre-log structures $\mathcal{LS}^+_{pre, X} $ is defined to be $\oplus_{\rho}\mathcal{N}_\rho$ where $\mathcal{N}_\rho$ is the extension by zero of the sheaf in Definition~\ref{def:pre_log}.
\end{definition}

Log structures will be sections of the sheaf $\mathcal{LS}^+_{pre, X}$ that satisfy a consistency condition that we now describe~\cite{GS1}.  Given a vertex $v \in \mathscr{P}$ fix:
\begin{itemize}
\item A cyclic ordering of the $1$-cells $\rho_i$ containing $v$;
\item Sections $f_i$ of $\mathcal{N}_{\rho_i}$; and
\item Dual vectors $\breve{d}_{\rho_i}$ annihilating the tangent spaces of $\rho_i$, and chosen compatibly with the cyclic ordering of $\rho_i$.
\end{itemize}
The consistency condition that we require is:
\[
\prod{\breve{d}_{\rho_i}\otimes_\Z f_i\big|_{V_v}} = 0\otimes 1
\]

\begin{remark}
In \cite{GS1} a further condition, \emph{local rigidity}, is imposed on $X_0(B,\mathscr{P},s)$ which, roughly speaking, is that the sections $f_i$ associated to the $1$-strata by the log structure do not factorize.  This is not a condition that we shall impose in our context.
\end{remark}

\begin{remark}
  Given a lattice polygon $Q$, we have constructed a family of affine manifolds $\mathcal{B}_{Q,k} \rightarrow \R_{\geq 0}$.  One could also consider the affine manifold of polygon type $(B, \mathscr{P})$ constructed from $Q$, and place a log structure on the scheme $X_0(B, \mathscr{P},s)$.    The choices involved in these two constructions are very closely related, as we now explain.

\begin{definition}\label{def:compatible}
Given any one parameter degeneration of affine manifolds $\pi \colon \mathcal{B} \rightarrow \R_{\geq 0}$ observe that any fiber $B$ of $\pi$ gives the same variety $X(B,\mathscr{P},s)$.  A one parameter family of log structures $s(x) \in \Gamma(\mathcal{LS}^+_{pre, X_0}) $, over $\C$ is said to be \emph{compatible with $\mathcal{B}$} if for each interior 1-cell $\tau$ and for each $x \in \C$ the following two subsets of $B$ coincide and have the same multiplicities:
\begin{enumerate}
\item The image of the zero set of the section $s(x)$ under the moment map.
\item The singular set $\Delta \subset B$, counted with multiplicity by singularity type.
\end{enumerate}
\end{definition}

\noindent Any one-parameter degeneration of affine manifolds $\pi \colon \mathcal{B} \rightarrow \R_{\geq 0}$ gives rise to a compatible one-parameter family of log structures.
\end{remark}

\section{Structures on Affine Manifolds}
\label{sec:structures}

In this section we define a \emph{structure} on $\left(B,\mathscr{P},\phi\right)$.  This is a purely combinatorial construction, which will encode the various functions used to reconstruct the formal deformation of the maximally degenerate variety $X_0(B, \mathscr{P},s)$.  This section is largely an exegesis of \cite{TropGeom}, Chapter 6.

\subsection{Exponents and orders}

Throughout this section we shall fix a triple $\left(B, \mathscr{P}, \phi\right)$ where:
\begin{enumerate}
\item $B$ is an affine manifold with singularities and corners.
\item $\mathscr{P}$ is a polygonal decomposition of $B$ into rational, convex polyhedra.
\item $\phi$ is a multi-valued piecewise linear function which is linear when restricted to full-dimensional cells.
\end{enumerate}

\begin{remark}
The multi-valued nature of $\phi$ reflects the fact that $B$ has singularities: $\phi$ may be defined as an affine function on the universal cover of $B \setminus \Delta$ but in general this will not take the same value on each point covering a given point $p \in B \setminus \Delta$.  Picking a sheet of the covering around $p$ is equivalent to making a choice of local representative for $\phi$.
\end{remark}

In view of this remark we shall define a sheaf twisted so as to ensure $\phi$ is a global section.  Formally, we shall define a sheaf of abelian groups on $B$ an extension by $\Z$ of $\Lambda$, the covariant lattice in the tangent space of $B$:
\[
0 \rightarrow \Z \rightarrow \mathcal{P}_\phi \rightarrow \Lambda \rightarrow 0
\]
To fix this sheaf we first choose an open over $U_i$ of $B_0$ and a representative $\phi_i$ of $\phi$ for each $U_i$:
\begin{definition}
The sheaf $\mathcal{P}$ is defined by taking $\mathcal{P}_\phi = \Lambda|_{U_i} \oplus \Z$ on restriction to each $U_i$.  On the intersection $U_i \cap U_j$ we identify sections via 
\[
(r,m) \sim (r+d(\phi_j-\phi_i)(m),m)
\]
 noting that $\phi_j-\phi_i$ is a linear function and so has a well defined slope which we evaluate in the direction $m$.
\end{definition}

\begin{definition}
An \emph{exponent} at $x \in B_0$ is an element of the stalk $\mathcal{P}_{\phi,x}$.
\end{definition}

\begin{definition}
There is a canonical projection $\mathcal{P}_{\phi,x} \rightarrow \Lambda_x$ for every $x \in B_0$.   Given an exponent $m \in \mathcal{P}_{\phi,x}$ we denote the image of $m$ under this projection by $\bar{m}$. 
\end{definition}

In the case where $B$ has no singularities, the deformations of the central fiber described in this section arise from a toric construction, which we now sketch (see \cite{TropGeom} for details).  The input data for this construction are an affine manifold $B \subset \R^2$, a decomposition $\mathscr{P}$ of $B$ into integral polygons and a convex function $\phi \colon B \rightarrow \R$ which is piecewise linear and linear on the elements of $\mathscr{P}$.  The set $B' = \{(p,x) : x \geq \phi(p)\}$ is a polyhedron, with a well defined normal fan.  The toric variety associated to this normal fan has a projection to $\C$ and the fiber over zero is equal to a reducible collection of toric varieties corresponding to the full-dimensional cells of $\mathscr{P}$.

\begin{example}
We consider a degeneration of $\P^1$: Let $B$ be the union of the intervals $[-1,0]$, $[0,1]$ and consider:
\[
\phi(x) =
\begin{cases} 
0 & x < 0 \\
x & x > 0
\end{cases}
\]
The toric variety associated to $B'$ is the blow up of $\C \times \P^1$ at $(0,\infty)$.  The projection onto the first factor has general fiber $\P^1$ and central fiber equal to the union of $2$ copies of $\P^1$ identified at a toric zero stratum.
\end{example}

\begin{remark} \label{rem:monoid_over_graph}
Observe that in this construction each cell of $\mathscr{P}$ not contained in the boundary of $B$ defines a cone via its tangent wedge in $B'$ which is dual to a cone in the normal fan of $B'$.  A chart of this degeneration is then given by taking the algebra over the monoid defined by the integral points of this tangent wedge.  
\end{remark}

We now localize this toric construction, so that it applies to $(B,\mathscr{P},\phi)$ such that $B$ has singularities.  In particular we shall define the analogue of the \emph{monoid above the graph} from Remark~\ref{rem:monoid_over_graph}.  To state this definition we need two more locally defined objects:
\begin{enumerate}
\item $\Sigma_x$: The fan in $T_xB_0$ induced by $\mathcal{P}$.
\item $\phi_{i,x}$: the piecewise linear function induced by $\phi_i$ on $T_xB_0$.  One may define this by defining its slope in each cell of $\Sigma_x$ to be the slope of $\phi_i$ in the cell of $\mathscr{P}$ that cone corresponds to; see \cite{TropGeom} for more details.
\end{enumerate}

\begin{definition}
Fix an $x \in U_i$.  We define a monoid $P_{\phi, x} \subseteq \mathcal{P}_{\phi, x}$ given by:
\[
P_{\phi, x} = \left\{(r,m) : \text{$m \in |\Sigma_x|$, $r \geq \phi_{i,x}\left(m\right)$} \right\}
\]
\end{definition}

\noindent The fact that $P_{\phi,x}$ is independent of the chart used to define it is proven in \cite{TropGeom}, and a corollary of that calculation is the following observation.

\begin{prop}
The order of an exponent with respect to a maximal dimensional cell $\sigma \in \mathscr{P}$ given by the formula $\ord_{\sigma}\left(p\right) = r - \phi_{i,\sigma}$ is independent of the chart used to define it.
\end{prop}

\noindent In words this definition is simply:
`\emph{The order of $m$ is its height above the hyperplane in $\mathcal{P}_{\phi,x}$ defined by $\sigma$}'. Thus we may extend the definition slightly:

\begin{definition}
For $\tau \in \mathscr{P}$ and $m \in |\Sigma|$, 
$\ord_{\tau}\left(m\right) = \max_{\tau \subseteq \sigma} \ord_{\sigma}\left(m\right)$ and
$\ord \left(m\right) = \max_\sigma \ord_{\sigma}\left(m\right)$.
\end{definition}

\subsection{Slabs and rays on $B$}

Structures on $B$ consist of a collection of \emph{slabs} and \emph{rays}.  We shall now define rays; these carry the instanton corrections analogous to gradient flow lines in \cite{KoSo}.  We recall this definition from \cite{TropGeom}.

\begin{definition}
A \emph{naked ray} (Definition 6.16 of \cite{TropGeom}) is an immersion $\d \colon \left[0,L_\d\right] \rightarrow B$ such that:
\begin{itemize}
\item whenever $\d(x)$ is non-singular, $D\d_x$ maps the integral tangent vectors to $x$ to $\Lambda_{\d(x)}$;
\item the image of $\d$ only intersects singular points in their monodromy invariant direction;
\item if $L_\d$ is finite then $\d\left(L_\d\right)$ is in $\partial B$.
\end{itemize}
A \emph{ray} is a pair $\left(\d,f_{\d}\right)$ where $\d$ is a naked ray, $f_\d = 1+c_mz^m$, and $m \in \Gamma\left( I_\d , \d^{-1}\mathcal{P}_\phi \right)$ is such that every germ $m_x$ of $m$ lies in $P_{\phi,\d\left(x\right)}$ 
\end{definition}

A crucial property of rays is that  the order of an exponent \emph{increases} as one moves from one cell of $\mathscr{P}$ to another; this follows from the strict convexity of the piecewise linear function $\phi$:

\begin{lem}
Consider a ray $\left(\d, f_{\d}\right)$ and the section $m$ giving the exponent of the ray function $f_\d$.  If $m_x \in P_{\phi, x}$ then for $x' > x, m_x' \in P_{\phi, x'}$.
\end{lem}
\begin{proof}
This is an immediate consequence of Lemma 6.19 in \cite{TropGeom}.
\end{proof}

\begin{remark}
This Lemma implies that given an integer $k$, the set
\[
\{ x \in [0,L_\d] : \ord_x\left(m\right) \leq k\}
\]
is an interval of the form $\left[0,N^k_\d\right]$; this defines the numbers $N^k_\d$ for each pair $\left(\d,k\right)$.  In particular we can define the truncation of a ray at a given order: \end{remark}

\begin{definition}
A $k$-truncated ray is a ray $\left(\d, f_{\d}\right)$ restricted to the domain $[0,N^k_{\d}]$.
\end{definition}

We now encode the log structure in the structure on $B$.  To do this we use a simplified version of the definition of a slab from \cite{GS1}.  We shall require the following preliminary observation:

\begin{lem}
Given a codimension one cell $\rho$ in $\mathscr{P}$ and a section $f_\rho \in \Gamma\left(V_\rho, \mathcal{O}\left(k\right)\right)$ defining the log structure along this stratum there is a canonical lift, which we also denote $f_\rho$, to a section of $k\left[P_{\phi,v}\right]$ for any vertex $v \in \rho$.
\end{lem}
\begin{proof}
The function $f_\rho|_{V\left(v\right)}$ is a polynomial function in $z^m$ where $m$ is the primitive generator of the tangent space to $\rho$.  Therefore $f_\rho|_{V\left(v\right)}$ is canonically an element of the ring $k\left[\Lambda_v\right]$.  We take $f_{\rho,v}$ to be the canonical lift to $\mathcal{P}_{\phi,v}$, obtained from the observation that $\phi$ gives a section of the projection $\mathcal{P}_{\phi, v} \rightarrow \Lambda_v$.  Notice that with respect to $\rho$ the order of the slab function is always zero.
\end{proof}

\begin{definition}\label{def:slab}
A \emph{slab} consists of a codimension one cell $\rho$ together with, for each non-singular point $x \in \rho$, a germ
\[
f_{\rho, x} = \sum_{m \in P_x, \bar{m} \in \Lambda_{\rho}}{c_mz^m} \in k\left[P_x\right]
\]
such that the following two conditions hold:
\begin{enumerate}
\item \emph{Change of vertex formula}: Take $x$ and $x'$ and denote the corresponding connected components of $\rho\setminus \Delta$ by $C_x$ and $C_{x'}$ respectively.   Let $k$ be the number of singularities (counted with multiplicity) between $x$ and $x'$, and define $m^\rho_{x,x'} \in \Lambda_x$ to be the $k$-fold dilate of the primitive generator of the ray from $x$ to $x'$.  Now we generalise the change of vertex formula of \cite{GS1} to give the relation between the slab functions in different connected components:
\[
f_{\rho,x'} = z^{m^\rho_{x,x'}}f_{\rho, x}
\]
\item \emph{Agreement with log structure}: If $x \in C_v$ for some vertex $v \in \rho$, we have at $v$ a function from the log structure: $f_\rho|_{V\left(v\right)}$.  There is a canonical parallel transport map to the point $x$ and we demand that, after parallel transport, we have $f_{\rho,x} = f_\rho|_{V\left(v\right)}$.
\end{enumerate}
\end{definition}

\begin{remark}
This definition of slab function relies on Proposition~\ref{prop:no_glue}.  Indeed the change of component formula in \cite{GS1} is considerably more complicated and it is not clear what the correct general definition is in cases which are not locally rigid.  
\end{remark}

\begin{remark}
In \cite{GS1} the authors ask only that the order zero part of the slab function agrees with the log structure; in \cite{TropGeom} however \emph{all} the corrections are carried by rays. Interpolating between these two, we shall regard slabs simply as placeholders for the log structure.
\end{remark}

\subsection{Defining a structure on $(B,\mathscr{P},\phi)$}

\begin{definition}
A structure $\mathscr{S} = \mathscr{S}^s \cup \mathscr{S}^r$ is a finite collection $\mathscr{S}^s$ of slabs and a possibly infinite collection $\mathscr{S}^r$ of rays such that:
\begin{enumerate}
\item The order of any exponent on any ray is strictly positive.
\item The set 
\[
\mathscr{S}_k^r = \left\{ \text{$k$-truncated rays $\left(\d,f_{\d}\right)$} : N^k_{\d} > 0 \right\}
\]
is finite for each $k$.
\end{enumerate}
\end{definition}

Given a structure $\mathscr{S}$ and a non-negative integer~$k$, we fix a polyhedral refinement $\mathscr{P}_k$ of $\mathscr{P}$ such that:
\begin{enumerate}
\item The cells of $\mathscr{P}_k$ are rational convex polyhedra.
\item For each $\d \in \mathscr{S}^r_k$, the set $\d\left(\left[0,N^k_{\d}\right]\right)$ is a union of cells in $\mathscr{P}_k$.
\end{enumerate}

We now define a category $\underline{\text{Glue}}(\mathscr{S},k)$ and a functor to the category of commutative rings which will record each of the local pieces of the smoothing.  This allows the problem of reconstructing the smoothing to be broken into two distinct problems: establishing functoriality, and then showing that the colimit of this functor produces a smoothing.

\subsubsection{The objects}

Let $\left(\omega,\tau,\mathfrak{u}\right)$ be a triple such that:
\begin{enumerate}
\item $\omega, \tau \in \mathscr{P}$ and a maximal cell $\mathfrak{u}$ of $\mathscr{P}_k$
\item $\omega \subseteq \tau$
\item $\omega \cap \mathfrak{u} \neq \varnothing$
\item $\tau \subseteq \sigma_\mathfrak{u}$, where $\sigma_\mathfrak{u}$ is the maximal cell of $\mathscr{P}$ containing $\mathfrak{u}$
\end{enumerate}

\begin{remark}
Each of these will be used to define a small subscheme of the formally degenerating family by considering a certain thickening of the stratum corresponding to $\tau$ inside a formal smoothing of $\Star(\omega)$.
\end{remark}

\subsubsection{The morphisms}

The space of morphisms between any two objects $\left(\omega, \tau, \mathfrak{u}\right),\left(\omega', \tau', \mathfrak{u}'\right)$ has at most one element.  It has one element precisely when $\omega \subseteq \omega'$, $\tau' \subseteq \tau$.  We shall use the following basic observation about the morphisms of this category:

\begin{lem}
Any morphism may be factored into morphisms of one of two types:
\begin{enumerate}
\item  $\omega \subseteq \omega'$, $\tau' \subseteq \tau$, $\mathfrak{u} = \mathfrak{u}'$.
\item  $\omega = \omega'$, $\tau' = \tau$, $\mathfrak{u}\cap\mathfrak{u}'$ is a one dimensional set containing $\omega$.
\end{enumerate}
\end{lem}

\noindent Note that this factorisation is generally non-unique.

\subsection{The gluing functor}

We now define the functor $F_k$ from $\textrm{\underline{Glue}}\left(\mathscr{S},k\right)$ to \underline{Rings} from which we shall construct the $k$th-order formal degeneration. The definition of this functor is virtually identical to that of \cite{TropGeom}.

Having fixed an object $\left(\omega, \tau, \mathfrak{u}\right)$ of $\textrm{\underline{Glue}}\left(\mathscr{S},k\right)$, we shall use the notation $\sigma$ for the maximal cell in $\mathscr{P}$ containing $\mathfrak{u}$.  We shall denote the ring $F_k(\omega, \tau, \mathfrak{u})$ by $R^k_{\omega, \tau, \mathfrak{u}}$;  $\Spec R^k_{\omega, \tau, \mathfrak{u}}$ is a thickening of the toric stratum corresponding to $\tau$. We give the definition of these rings in three stages.

\subsubsection{Defining $P_{\phi, \omega}$}

Recall the monoid $P_{\phi, x}$ for $x \in \textrm{Int}\left(\omega\right)$.  If we pick a $y \in \sigma$ then since the interior of a cell in $\mathscr{P}_{max}$ is simply connected there is a well-defined inclusion $j\colon P_{\phi, x} \hookrightarrow \mathcal{P}_{\phi, y}$ via parallel transport.

\begin{definition}
$P_{\phi, \omega} = j\left(P_{\phi, x}\right) \subseteq \mathcal{P}_{\phi, y}$.
\end{definition}

\subsubsection{Defining the ideal $I^k_{\omega, \tau, \sigma}$}

The thickening of the stratum is defined by an ideal,
$I^k_{\omega, \tau, \sigma} = \left\{ m \in P_{\phi, \omega} : \ord_{\tau}\left(m\right) > k \right\}$. We set $R^k_{\omega \tau \sigma} = k\left[P_{\phi, \omega}\right]/I^k_{\omega \tau \sigma}$.

\subsubsection{Localisation}

This is not yet a good enough  definition of $F_k(\omega, \tau, \sigma)$ however.  The change of vertex formula in the definition of slab demands that certain functions (which have zeroes on the toric 1-strata) should be invertible in these rings, therefore we need to localise with respect to these functions.  This is broken into cases, depending on the strata $\omega$,~$\tau$.

First assume that $\tau$ is an edge with non-trivial intersection with $\Delta$.  In this case we have a slab function attached to each smooth point of $\tau$, and we form the localisation:

\begin{definition}
$R^k_{\omega \tau \mathfrak{u}} = \left(R^k_{\omega \tau \sigma}\right)_{f_{\tau}}$
\end{definition}

Precisely, we need to specify what $f_{\tau}$ means here.  If $\omega = \tau$ it is irrelevant, the slab function is a polynomial in a single variable which is invertible in this ring.  If $\omega$ is a vertex we simply take the germ of the slab function at this point.

In all other cases, namely $\tau \cap \Delta = \varnothing$, we define:

\begin{definition}
$R^k_{\omega \tau \mathfrak{u}} = R^k_{\omega \tau \sigma}$ 
\end{definition}

\noindent We are now able to define the functor $F_k$ on objects: \[F_k\left(\omega,\tau,\mathfrak{u}\right) = R^k_{\omega,\tau,\mathfrak{u}}\]

\begin{remark}
We observe there are some canonical maps between various of these rings.  If $\tau' \subseteq \tau$ and $\omega \subseteq \omega'$ there is a canonical inclusion
$ I^k_{\omega,\tau,\sigma} \hookrightarrow I^k_{\omega,\tau',\sigma} $ and thus a surjection $R^k_{\omega,\tau,\sigma} \rightarrow R^k_{\omega,\tau',\sigma}$.  There is also an inclusion of monoids $P_{\phi, \omega, \sigma} \hookrightarrow P_{\phi,\omega',\sigma}$ and thus an injection $R^k_{\omega,\tau,\sigma} \hookrightarrow R^k_{\omega',\tau,\sigma}$.  One may check that these maps survive the localisations at the slab functions.
\end{remark}

Now we have defined the functor on objects we define the functor on morphisms.  This is done case by case, recalling that any morphism may be factored into those of \emph{change of strata} type and those of \emph{change of chamber} type.

\subsubsection{Change of strata}

We specify a map:
\[
R^k_{\omega,\tau,\sigma} \hookrightarrow R^k_{\omega' ,\tau',\sigma}
\]
by composing the canonical maps we identified in the previous section, precisely, we define the change of strata map:
\[
\psi_{(\omega,\tau),(\omega',\tau')} \colon R^k_{\omega, \tau, \mathfrak{u}} \rightarrow R^k_{\omega, \tau', \mathfrak{u}} \hookrightarrow R^k_{\omega', \tau', \mathfrak{u}}
\]
to be the composition of the two maps above. See \cite{TropGeom} for the verification that these are defined in the localised rings.

\subsubsection{Change of chamber maps}

Now we fix two chambers $\mathfrak{u}$,~$\mathfrak{u}'$ with one dimensional intersection and such that $\omega \cap \mathfrak{u} \cap \mathfrak{u}' \neq \varnothing$.  We also fix a point $y \in \textrm{Int}\left(\mathfrak{u} \cap \mathfrak{u}'\right)$ such that the connected component of $B_0 \cap \mathfrak{u} \cap \mathfrak{u}'$ (recalling $B_0 := B \backslash \Delta$) containing $y$ intersects $\omega$.  Note that either $\omega$ is a vertex, in which case there is a unique such component, or $\omega$ is an edge, in which case any connected component will do.  We shall now define the change of chamber map $\theta_{\mathfrak{u}, \mathfrak{u}'}\colon R^k_{\omega, \tau, \mathfrak{u}} \rightarrow R^k_{\omega, \tau, \mathfrak{u}'}$.

We consider two further cases, depending on whether or not $\sigma_\mathfrak{u} \cap \sigma_{\mathfrak{u}'} \cap \Delta = \varnothing$.  If this is the case we define:
\[
\theta_{\mathfrak{u}, \mathfrak{u}',y}\left(z^m\right) = z^m \prod{f^{\left\langle n,\bar{m}\right\rangle}_{\left(\d,x\right)}}
\]
Note that this is always an isomorphism -- all the functions $f_{\left(\d,x\right)}$ are invertible.  As rays propagate in the direction of $\bar{m}$ this is manifestly independent of the point $y$.  If $\sigma_\mathfrak{u} \cap \sigma_{\mathfrak{u}'} \cap \Delta = \varnothing$, we shall define the map as follows:
\[
\theta_{\mathfrak{u}, \mathfrak{u}',y}\left(z^m\right) = z^m f^{\left\langle n,\bar{m}\right\rangle}_{\rho,y}\prod{f^{\left\langle n,\bar{m}\right\rangle}_{\left(\d,x\right)}}
\]
\begin{remark}
Notice that $z^m$ in the left hand side is an element of $R^k_{\omega, \tau, \mathfrak{u}}$ whereas on the right it appears as an element of  $R^k_{\omega, \tau, \mathfrak{u}'}$.  The identification of these two rings is made via parallel transport along a `short path' from $\mathfrak{u}$ to $\mathfrak{u'}$ which is contained in the union of these two chambers and which intersects the 1-cell between them only once.
\end{remark}

\noindent Since $R^k_{\omega, \tau, \mathfrak{u}}$ is localised at the slab functions we see that all functions appearing in the product are invertible, and so this map is an automorphism.  However, the above definition is not manifestly independent of $y$.

\begin{prop}
$\theta_{\mathfrak{u}, \mathfrak{u}',y}$ is independent of the choice of $y$.
\end{prop}
\begin{proof}
Since this is proven in \cite{TropGeom} we only provide a sketch of this proof.  The key observation is that if we change from $y$ to $y'$ in a different component of $\mathfrak{u} \cap \mathfrak{u} \cap B_0$ we change the slab function by the transition function given in Definition~\ref{def:slab}.  However we also change the identification of this stalk with $R^k_{\omega,\tau,\mathfrak{u}'}$  by parallel transport, which may be interpreted as precomposing this map with the isomorphism induced by a simple loop around the singular point.  The factors in these two isomorphisms are the same, but occur with different signs, ensuring that the change of path does not alter the change of chamber map.
\end{proof}

\subsubsection{Functoriality}

We have now defined a map on objects and on `elementary' morphisms; however we need to show both that this is well defined and that this is a functor.  We first define a \emph{joint} which will be used to formulate a necessary and sufficient condition for functoriality:

\begin{definition}
A vertex of $\mathscr{P}_k$ not contained in the boundary of $B$ is called a \emph{joint}.  The collection of joints of $\mathscr{P}_k$ is denoted $\textrm{Joints}\left(\mathscr{S},k\right)$.
\end{definition}

\noindent Indeed, fixing a $\mathfrak{j} \in \textrm{Joints}\left(\mathscr{S},k\right)$ and a cyclic ordering $\mathfrak{u}_1, \cdots ,\mathfrak{u}_k$ of the chambers around this vertex one has a necessary condition for $F_k$ to be a functor:
\begin{equation}\label{eq:consistent}
\theta_{\mathfrak{u}_1,\mathfrak{u}_2} \circ \cdots \circ \theta_{\mathfrak{u}_k,\mathfrak{u}_1} = \text{Id}
\end{equation}

\noindent The content of Theorem 6.28 of \cite{TropGeom} is that it is sufficent to check this identity at every joint.  Given what have said already, this is a purely formal exercise and the reader is referred to \cite{TropGeom} for the proof of this result.

\begin{definition}
Given a structure $\mathscr{S}$ and a joint $\mathfrak{j}$ we say $\mathscr{S}$ is \emph{consistent at $\mathfrak{j}$ to order $k$} if and only if Equation~\ref{eq:consistent} holds at $\mathfrak{j}$ to order $k$.  $\mathscr{S}$ is called \emph{compatible to order $k$} if it is consistent to order $k$ at every joint.
\end{definition}

By Theorem 6.28 of \cite{TropGeom} compatibility of the structure $\mathscr{S}$ implies the existence of a well defined functor from the category $\textrm{\underline{Glue}}\left(\mathscr{S},k\right)$ to \underline{Rings}.

\section{Consistency and Scattering}
\label{sec:scattering}

We saw in the last section that in order for the gluing functor to be well defined we need to guarantee a consistency condition on the structure.  In this section we shall describe an inductive algorithm for ensuring this is the case at each order.  Theorem 6.28 of \cite{TropGeom} has reduced this to a local computation at each joint.  Indeed, fixing a joint $\mathfrak{j}$ we shall construct a \emph{scattering diagram} $\mathfrak{D}_{\mathfrak{j}}$ which will encode this local data.  We begin by outlining the necessary theory associated with scattering diagrams.

\subsection{Scattering diagrams at joints}

This section is based on Section 6.3.3 of \cite{TropGeom} and on \cite{GPS}.  This section is also largely independent of the rest of the article; we can make these definitions independently of a structure $\mathscr{S}$ or an affine manifold $B$.

We shall fix the following data:
\begin{enumerate}
\item A lattice $M \cong \Z^2$, and denote $N = \textrm{Hom}_{\Z}\left(M,\Z\right)$.
\item $P$ a monoid, and a map $r\colon P \rightarrow M$.  We shall denote $\mathfrak{m} := P \backslash P^{\times}$.
\end{enumerate}

\noindent The scattering diagram itself will consist of a number of \emph{rays} and \emph{lines}:

\begin{definition}
A \emph{ray} (resp. \emph{line}) is a pair $\left(\mathfrak{d}, f_{\mathfrak{d}}\right)$.  Here $\mathfrak{d} = m_0' - \R_{\geq 0}m_0$ for a ray (resp. $\mathfrak{d} = m_0' - \R m_0$ for a line).  Viewing $\mathfrak{d}$ as a set gives the \emph{support} of the ray (line).  If $\d$ is a ray we call $m_0'$ the \emph{initial point}.  The function $f_{\d}$ is an element of $\widehat{k\left[P\right]}$, with the completion taken with respect to $\mathfrak{m}$, such that:
\begin{itemize}
\item $f_{\d}$ is congruent to one modulo the maximal ideal, i.e. $f_{\d} \in 1 \bmod \mathfrak{m}$
\item $f_{\d}$ may be written $f_{\d} = 1 + \sum{c_mz^m}$ such that if $c_m \neq 0$, $r(m) = Cm_0$ for a positive rational number $C$.
\end{itemize}
\end{definition}

\begin{definition}
A \emph{scattering diagram} $\mathfrak{D}$ over $k[P]/I$ is a finite collection of rays and lines such that $f_{\mathfrak{d}} \in k[P]$.
\end{definition}

Given a ray or a line $\left(\d,f_{\d}\right)$ we define an automorphism of $k[P]/I$ as follows:

Fix a path $\gamma$ that intersects $\d$ transversely and a primitive element $n \in N$ annihilating the support of the ray such that the direction $n$ is compatible with the orientation of the $\gamma$.

Given these choices, set $\theta_{\gamma,\d}\left(z^m\right) = z^mf_{\d}^{\left\langle n,r\left(m\right)\right\rangle}$.  Composing these in sequence we can describe automorphisms arising from longer paths, or indeed loops, forming the \emph{path ordered product} associated with these paths.  Specifically, given a path $\gamma$ we may define $\theta_{\gamma,\mathfrak{D}} = \theta_{\gamma,\d_1} \cdots \theta_{\gamma,\d_n}$ so long as $\gamma$ intersects each of the $\d_i$ transversely at time $t_i$, with $t_i > t_{i+1}$, and avoids the intersection points of any rays or lines.

\begin{remark}
One may equivalently define the wall crossing automorphism $\theta_{\gamma,\d}$ by considering the element $f_\d\partial_n$ of the Lie algebra of log derivations.  The element $\theta_{\gamma,\d}$ of $\textrm{Aut}\left(k[P]/I\right)$ is obtained by exponentiation from this Lie algebra.  For more details the reader is referred to \cite{GPS}.
\end{remark}

There is a natural notion of consistency for a scattering diagram:

\begin{definition}
A scattering diagram $\mathfrak{D}$ is \emph{consistent} if and only if the path ordered product around any loop for which this product is defined is the identity in $\textrm{Aut}\left(k[P]/I\right)$.
\end{definition}

One fundamental property of scattering diagrams is that one may add rays in an essentially unique fashion to achieve consistency.  This is the content of the following result of Kontsevich--Soibelman:
\begin{thm}
Given a scattering diagram $\mathfrak{D}$, then there is a scattering diagram $\textrm{S}_I\left(\mathfrak{D}\right)$ such that $\textrm{S}_I\left(\mathfrak{D}\right) \backslash \mathfrak{D}$ is entirely rays, and is consistent over the ring $k[P]/I$.
\end{thm}

\begin{proof}
The proof is a calculation in the Lie algebra of log derivations and the subalgebra which exponentiates to the tropical vertex group.  This is discussed in much more detail in \cite{GPS}.
\end{proof}

We now have a framework in which we can introduce corrections to order $k$, inductively making a scattering diagram consistent.  Recalling that we have fixed a joint $\mathfrak{j}$ in $\mathscr{S}$ on $\left(B,\mathscr{P},\phi\right)$ we fix the data required to define a scattering diagram:
\begin{definition}
Let the lattice be $M = \Lambda_{\mathfrak{j}}$, the monoid $P = P_{\phi, \sigma_{\mathfrak{j}},\sigma}$ and the map $r\colon P \rightarrow M$ be given by $m \mapsto \bar{m}$.  Noting that in general we have a maximal ideal $\mathfrak{m} = P\backslash P^{\times}$ we fix an $\mathfrak{m}\textrm{-primary ideal}, I = I^k_{\sigma_j, \sigma_j,\sigma}$.
\end{definition}

We construct the scattering diagram $\mathfrak{D}_{\mathfrak{j}}$ in two steps.
\begin{enumerate}
\item If $\mathfrak{j} \subset \rho$ where $\rho$ is a slab, that is $\rho \cap \Delta \neq \varnothing$, then we factorize $f_{\rho,x}$ for $x \in \rho$, writing $f_{\rho,x} = \prod_j{1+c_{\rho,j}z^{l_jm_{\rho,x}}}$.  For each $j$ we add the following line to the scattering diagram:
\[
\left( \R m , 1 + c_{\rho, j}z^{l_jm_{\rho,x}}\right)
\]
where $m$ is the primitive vector in the direction of $T_x\rho$.
\item For each ray $\d$ in $\mathscr{S}_{k-1}$ such that there exists $x \in \left[0,N^k_{\d}\right]$ with $\d(x) \in \mathfrak{j}$ we add either a ray or a line. If $x = 0$ we add a ray:
\[
\left( \R_{\geq 0}\mathfrak{d}'\left(x\right) , 1 + c_{\d}z^{m_\d,x}\right)
\]
otherwise we add the line with the same function.
\end{enumerate}

Section 6.3.3 of \cite{TropGeom} establishes that if $\dim \sigma_j  \in \left\{0,2\right\}$ then in fact $\mathfrak{D}_j$ satisfies all the requirements of a scattering diagram and so one may apply the Kontsevich--Soibelman algorithm and obtain a consistent scattering diagram $S_I\left(\mathfrak{D}_j\right)$.  The rays of $S_I\left(\mathfrak{D}_j\right)$ are then `exponentiated' to give rays locally in the structure $\mathscr{S}$ which then propagate in $B$.

Of course we have not dealt with the case that $\dim \sigma_j = 1$.  This is harder because the candidate scattering diagram does not satisfy the requirement that $f_{\d} \in 1 \bmod \mathfrak{m}$ for those lines coming from the slabs.  Indeed, those functions always have order zero in the interior of $\rho$.  A solution would be to try and prove an analogue of the Kontsevich--Soibelman Lemma over the localised ring $\left(k[P]/I\right)_{f_{\rho,x}}$. However, the approach taken in \cite{TropGeom} is to work in an even larger ring, define a `universal' scattering diagram and view the localised ring as a subring.  Since we impose slightly weaker assumptions on the singular locus $\Delta$ than appear in \cite{TropGeom} we require a slightly stronger result, which is the topic of the next section.

\subsection{Localising scattering diagrams}

This section details the required modest amendments to Proposition 6.47 of \cite{TropGeom} needed in order to extend that result to `non-simple' settings. Roughly, by replacing coefficients with formal variables one may embed the localised ring in a completion of the original ring with respect to a sequence of ideals $I_e$.  Once one can show that the scattering diagrams $\text{S}_{I_e}\left(\mathcal{D}\right)$ stabilize we may form the scattering diagram over this completed ring.

Before stating the proposition we require some results from \cite{TropGeom} relating scattering diagrams and enumerative geometry.  To state these we first consider a scattering digram of the following form:
\begin{equation}\label{eq:scattering_form}
\mathfrak{D} = \left\{\R m_i ,\left( \prod_{j=1}^{p_i}\prod^{l_{ij}}_{k=1}\left(1+t_{ijk}z^{-jm_i}\right)\right) : 1 \leq i \leq p \right\}
\end{equation}
Starting with this scattering diagram we shall study $S\left(\mathfrak{D}\right)$, over the ring $k[M][\![ \left\{t_{ijk}\right\} ]\!]$.  Note that we can always reduce by an $\mathfrak{m}$-primary ideal $I$, to form $S_I\left(\mathfrak{D}\right)$.
We further assume that no two rays have the same support and fix a ray $\left(\d,f_\d\right) \in S\left(\mathfrak{D}\right) \backslash \mathfrak{D} $.  Reducing mod $I$ we can assume that $f_\d$ is a polynomial.  We now construct a toric variety corresponding to $\d$:

\begin{definition}
Let $X_{\d}$ be the non-singular toric surface associated to the complete fan $\Sigma_{\d}$ which includes the rays: $\R_{\geq 0}m_i$ and $\d$ for each $m_i$ in the definition of the scattering diagram above.  Let $D_i$ denote the toric divisor corresponding to $m_i$ and let $D_{\textrm{out}}$ denote the toric divisor corresponding to $\d$.
\end{definition}

We also need some auxiliary combinatorial definitions to state an enumerative formula for $f_{\d}$:

\begin{definition}
A graded partition $G$ is a finite sequence $G = \left(P_1,\cdots, P_d\right)$ of ordered partitions $P_i = \left(p_{i1},\cdots, p_{il_i}\right)$, where $i\mid p_{ij}$ for each $i$ and $j$.  We call $p_{ij}$ the \emph{parts} of $P_i$ and define $|P_i| = \sum_j{p_{ij}}$ and $|G| = \sum{|P_i|}$.
\end{definition}

Now let $G = \left(G_1, \cdots G_p\right)$ be a tuple of graded partitions, where we denote by $P_{ij}$ the $j$th piece of $G_i$ and write $P_{ij} = \left(p_{ij1}, \cdots , p_{ijl_{ij}} \right)$.

As in \cite{TropGeom} restrict to those $G$ such that
\begin{equation}\label{eq:partition_condition}
-\sum{|G_i|m_i} = k_Gm_\d
\end{equation}
for some $k_G \in \Z_{>0}$.  Now fix the class $\beta \in H_2\left(X_\d,\Z\right)$ such that:
\begin{enumerate}
\item If $D \notin \left\{D_1,\cdots D_p,D_{\textrm{out}}\right\}$ then $\beta.D = 0$
\item $\beta.D_i = |G_i|$
\item $\beta.D_{\textrm{out}} = k_G$
\end{enumerate}
If $D_{\textrm{out}} = D_i$ for some $i$ replace the above prescription of $\beta.D_i$ with $\beta.D_i = |G_i| + k_G$.  Next pick general points $x_{ijk}$ on $D_i$ and recall the notion of an orbifold blowup from \cite{GPS}:
\begin{definition}
Let $p \in D$ be a point in a non-singular divisor in a surface $S$.  There is a unique length $j$ subscheme supported at $p$. Let $\mathcal{S}_j \rightarrow S_j \rightarrow S$ be the composition of the blowup map in this ideal sheaf and the coarse moduli map from the unique orbifold structure on the singular variety $S_j$.
\end{definition}
\begin{remark}
The exceptional divisor $E$ in the blown-up space has self intersection $\left[E\right]^2 = -1/j$
\end{remark}

\noindent We now define a space by making the orbifold blow-ups designated by $G$.

\begin{definition}
Let $\nu \colon X\left[G\right]\rightarrow X$ be the length $j$ orbifold blow-up of $X$ in each of the points $x_{ijk}$.
\end{definition}

\noindent We shall use a Gromov-Witten invariant associated to the strict transform:
\[
\beta_G = \nu^*\left(\beta\right) - \sum_{ijk}{p_{ijk}\left[E_{ijk}\right]}
\]
Colloquially this is the virtual number of rational curves with tangency order $k_G$ along $D_\textrm{out}$ at exactly one point, and $p_{ijk}/j$ branches tangent to $D_i$ with order $j$ at $x_{ijk}$.  The precise definition is an integral over a moduli space of stable relative maps with orbifold target space $X^o_\d$; see \cite{GPS}.  Here, conforming to the notation of \cite{GPS}, $X^o_\d$ is the space obtained by removing the toric zero-strata from $X_{\d}$.  We call the result of the blow-up $\nu$, $\widetilde{X}^o_{\d}$.

Theorem 6.44 of \cite{TropGeom} describes $\log\left(f_{\d}\right)$ in terms of these Gromov-Witten invariants:
\begin{thm}
\[
\log\left(f_\d\right) = \sum_G k_GN_Gt^Gz^{-k_Gm_\d}
\]
where $t^G = \prod{t_{ijk}^{p_{ijk}/j}}$ and the sum is over graded partitions $G$ satisfying Equation~\ref{eq:partition_condition}.
\end{thm}

\noindent We also recall Remarks 6.45 and 6.46 of \cite{TropGeom}:

\begin{remark}
The definition of relative stable maps includes the possibility of maps $f\colon C \rightarrow \widehat{X}^o_{\d}$ to a reducible scheme, but $\widehat{X}^o_{\d}$ comes with a map to $\widetilde{X}^o_{\d}$ and thus fits into a diagram:

\begin{displaymath}
    \xymatrix{
    C \ar[r]^f \ar@/^2pc/[rr]^{\tilde{f}} \ar[drr]_{\bar{f}}& \widehat{X}^o_{\d} \ar[r] & \widetilde{X}^o_{\d} \ar[d]^{\nu} \\
    & & X^o_{\d}}
\end{displaymath}

\noindent Results cited in \cite{TropGeom} imply that $\tilde{f}\left(C\right)\cap \tilde{D}^o_i = \varnothing$.  We can now make statement about the intersection properties of $\bar{f}_*[C]$.  In particular as this represents $\beta_G$ the intersection multiplicity at each of the points $x_{ijk}$ must be exactly $p_{ijk}$.  Futher there is a point $q \in D_{\textrm{out}}$ such that $\bar{f}_*\left[C\right] \cap \partial X_{\d} = \left\{x_{ijk}\right\} \cup \left\{q \right\}$, and this point is constrained to lie on one of finitely many points of $D_{\textrm{out}}$.  The full argument is in \cite{TropGeom}, but in short one can describe the restriction of $\bar{f}_*\left[C\right]$ to $\partial X_{\d}$ in terms of $q$, but this is in the linear equivalence class given by $\beta|_{\partial X_{\d}}$, so only those values of $q$ which will land in this equivalence class are permitted.
\end{remark}

We now relate general scattering diagrams to the apparently special type we described above.

\begin{remark}
A general scattering diagram consisting solely of lines is equivalent to one of the form:
\[
\mathfrak{D} = \left\{\left(\R\bar{m}_i, \prod_{j,k}\left(1+c_{ijk}z^{-m_{ijk}}\right)\right): 1 \leq i \leq p\right\}
\]
such that $r\left(m_{ijk}\right)$ is proportional to $\bar{m}_i$ with index $j$.  We now define a scattering diagram of the form considered in \ref{eq:scattering_form}:
\[
\mathfrak{D}' = \left\{\left(\R\bar{m}_i, \prod_{j,k}\left(1+t_{ijk}z^{-r\left(m_{ijk}\right)}\right)\right) : 1 \leq i \leq p\right\}
\]
This is now a scattering diagram over $k[M]\llbracket\{t_{ijk}\}\rrbracket$.  Thus we have an enumerative interpretation for the rays of $S(\mathfrak{D}')$.  We shall refer to this as a `universal scattering diagram'.  Rather than defining $\mathfrak{D}'$ over all $k[M]\llbracket\{t_{ijk}\}\rrbracket$ we can consider the monoid $Q \subseteq M \oplus \N^l$ where the second factor corresponds to the $t_{ijk}$ variables and $l = \sum_{i,j}l_{ij}$.  There is a ring homomorphism $\phi\colon t_{ijk}z^{-r\left(m_{ijk}\right)} \mapsto c_{ijk}z^{-m_{ijk}}$ and we can define $\mathfrak{D}'$ over $k[Q]/\phi^{-1}(I)$ for an $\mathfrak{m}\textrm{-primary ideal}$ $I$.  Following \cite{TropGeom} we observe that there is a scattering diagram $\phi\left(S_{I'}\left(\mathfrak{D}'\right)\right)$ which is equivalent to $S_I(\mathfrak{D})$.
\end{remark}

\begin{remark}
Given a joint $\mathfrak{j}$ supported on the interior of a 1-cell $\tau$ we may write down a collection of rays and lines as for a scattering diagram; we refer to this collection of rays and lines as $\mathfrak{D_j}$ and write: 
\[
\overline{\mathfrak{D}}_{\mathfrak{j}} = \left\{ \d \in \mathfrak{D_j} : \d \textrm{ is a line} \right\}
\]
By rewriting and factorising the functions attached to the slab and rays intersecting this joint we may assume $\overline{\mathfrak{D}}_{\mathfrak{j}}$ is of the form:
\[
\overline{\mathfrak{D}}_{\mathfrak{j}} = \left\{ \R\bar{m}_i , \prod_{j,k}{\left(1+c_{ijk}z^{-m_{ijk}}\right)} \right\}
\]
Deviating from \cite{TropGeom}, there may be several factors (not just one) which are not in the maximal ideal $\mathfrak{m}$. 
\begin{definition}
Notice that we have factorized the slab function at $\mathfrak{j}$; consequently we may define a set $\mathcal{J}$ of triples $(i,j,k)$ such that:
\[
f_{\tau,\mathfrak{j}} = \prod_{\mathcal{J}}{\left(1+c_{ijk}z^{-m_{ijk}}\right)}
\]
\end{definition}

\noindent Recall that $i$ here indexes the direction vectors of rays, and that any `bad factor' (that is, any factor not of the form $1+x$ with $x \in \mathfrak{m}$) is associated to the (one-dimensional) slab $\tau$. Thus if $(i,j,k) \in \mathcal{J} $ then $i$ must be one of at most two possibilities.  If there are two distinct values of $i$ denote them $i_+$,~$i_-$ and note that $\bar{m}_{i_-} = -\bar{m}_{i_+}$.  Conversely, if all the elements of $\mathcal{J}$ have a unique value of $i$ then we shall refer to this as $i_+$ and shall not define $i_-$.

We shall define an inverse system of ideals $I_e \subseteq k[M]\otimes_k k[\{t_{ijk}\}]$ such that $\overline{\mathfrak{D}}_{\mathfrak{j}}$ is a genuine scattering diagram with respect to each $I_e$ and use the enumerative interpretation of these scattering diagrams to show these stabilise as $e \rightarrow \infty$.  This will imply that we can define a scattering diagram over the completion with respect to the inverse system $I_e$; the required localisation is then a subring of this completion.
\end{remark}

\begin{prop}\label{prop:stab}
Consider $J$ a monomial ideal in the ring
\[
R = k\left[t_{ijk} : (i,j,k) \notin \mathcal{J} \right]
\]
with $R/J$ artinian.  For a non-negative integer $e$, let
\[
I_e = \sum_{\left(i,j,k\right) \in \mathcal{J}}\left(t_{ijk}^e\right) + J
\]
in $k[M]\otimes_kk\left[\left\{t_{ijk}\right\}\right]$.  Now apply the Kontsevich--Soibelman algorithm to obtain $S_{I_e}\left(\mathfrak{D}\right)$ and remove all $\left(\d,f_{\d}\right)$ equal to 1 modulo $I_e$.  The sequence of scattering diagrams $\mathfrak{D}_1, \mathfrak{D}_2, \cdots $ stabilizes.
\end{prop}
\begin{proof}

Take $\Gamma$ to be the set of collections of graded partitions $G = \left(G_1,\cdots, G_p\right)$ such that:
\[
\prod_{\left(i,j,k\right) \notin \mathcal{J}}{t_{ijk}^{p_{ijk}/j}} \notin J
\]
and such that $p_{ijk}>0$ for some $(i,j,k) \notin \mathcal{J}$.  Now $R/J$ Artinian implies that \emph{having fixed} the values of $\left\{p_{ijk} : (i,j,k) \in \mathcal{J} \right\}$ there are finitely many choices of $G$, but these are themselves unconstrained.  We proceed in two steps, following Proposition 6.47 in \cite{TropGeom}. First we show that there are only a finite number with $N_G \neq 0$, then we bound the number of terms of any $\log f_\d$ independently of $e$.

Suppose $G \in \Gamma$ and $N_G \neq 0$. Then there is a primitive integral vector $m_\d$ such that:
\[
-\sum_{i}|G_i|m_i = k_Gm_\d
\]
and such a ray $(\d,f_\d)$ must appear in the scattering diagram $\mathfrak{D}_e$ with support $\R_{\geq 0} m_\d$.  Let $\Sigma_\d$ be as above; recall this fan is only determined up to arbitrary fan refinements so we may assume that both $\R_{\geq 0} m_{i_+}$ and $\R_{\leq 0}m_{i_+}$ appear in this fan, noting that the latter is equal to $\R_{\geq 0}m_{i_-}$ if $i_-$ is defined.  Hence there exists a toric morphism:
\[
\pi \colon X_\d \rightarrow \P^1
\]
defined by these two rays.  There are two toric sections of this morphism, which we shall refer to as $D_+$ and $D_-$ corresponding to $\R_{\geq 0}m_{i_+}$ and $\R_{\leq 0}m_{i_+}$ respectively.  Now $N_G \neq 0$ implies there is a map $\bar{f}\colon C \rightarrow X_\d$ such that $\bar{f}_*\left(C\right)$ has intersection multiplicity $p_{ijk}$ at $x_{ijk}$ for $\left(i,j,k\right) \in \mathcal{J}$.  Without loss of generality we assume that for any $t \neq 0$,~$\infty$ the fiber $\pi^{-1}(t)$ contains at most one of the points $x_{ijk}$.

We wish to eliminate the possibility that the image of $\bar{f}$ contains $\pi^{-1}\left(\pi\left(x_{ijk}\right)\right)$ for any $(i,j,k) \in \mathcal{J}$.  Observe that $\pi^{-1}\left(\pi\left(x_{ijk}\right)\right)$ meets $\partial X_\d$ at a point other than any $x_{ijk}$; call this point $q'$.  We also know that the divisor class $\sum{p_{ijk}x_{ijk}} + k_Gq$ is of the class $\beta|_{\partial X_\d}$ which is determined by $G$.  Indeed, the set $\bar{f}(C) \cap \partial X_\d$ is the collection $x_{ijk}$ and one additional point $q$.  If we assume that $\bar{f}(C)$ contains this fibre $\pi^{-1}\pi(x_{ijk})$, then we must have that $q=q'$.  However we have assumed that there is at least one $x_{ijk}$ such that $(i,j,k) \notin \mathcal{J}$ and $p_{ijk} > 0$, moving this point alone we obtain a contradiction.

As remarked, $\bar{f}_*\left(C\right)$ represents the class $\beta$ and $\tilde{f}_*\left(C\right)$ represents $\beta_G$, the strict transform defined above. The total transform of $\pi^{-1}\left(\pi\left(x_{ijk}\right)\right)$ contains the irreducible component $E_{ijk}$, and we know that $E_{ijk}.\beta_G = p_{ijk}$.  Thus if $F$ is the class of the fiber of $\pi$, $\beta.F \geq p_{ijk}$.  Now assume $\pi^{-1}(0)$ does not contain $D_{\textrm{out}}$, indeed, swap it with $\pi^{-1}\left(\infty\right)$ if it does. The proper transform of $\pi^{-1}(0)$ is disjoint from $\tilde{f}(C)$ but $\beta.\pi^{-1}\left(0\right)$ is determined by the $G_i$ for $i \neq i_+$,~$i_-$.  Thus $p_{ijk}$ is bounded and this bound is independent of $\d$, so there are a finite number of possiblities for $G \in \Gamma$.

The rest of the proof of Proposition 6.47 in \cite{TropGeom} goes through as stated, expect that now we need to observe that
\[
\left\{\left(\R m_i, 1+t_{ijk}z^{-m_i}\right) : \left(i,j,k\right) \in \mathcal{J} \right\}
\]
contains no rays, meaning that the formula for any ray must have a coefficent $t_{ijk}$ for some $\left(i,j,k\right) \notin \mathcal{J}$, and so the number of terms appearing in the formula for $\log(f_\d)$ is finite, and with bound determined by $J$, that is independent of $e$.  One can now apply a factorization process and generate rays with functions $f_\d$ all of the form $1+cz^m$.

\end{proof}

As remarked in \cite{TropGeom} the purpose of this result to form $ S\left(\mathfrak{D}\right) = \cup S_{I_e}\left(\mathfrak{D}\right)$, which is a scattering diagram over the completion of $A = \C[M]\otimes k\left[\{t_{ijk}\}\right]$ with respect to $\sum_{\left(i,j,k\right) \in \mathcal{J}}{\left(t_{i,j,k}\right)}$; this completion contains the subring given by localising $A$ at the various factors $1+t_{ijk}z^{-m_{ijk}}$.

We can now generate rays in the structure $\mathscr{S}$ from the rays of this scattering diagram, yielding a compatible structure:
\begin{thm}
$\mathscr{S}_k$ is compatible to order $k$.
\end{thm}
\begin{proof}
The proof of Theorem 6.49 in \cite{TropGeom} now goes through exactly, replacing Proposition 6.47 there with Proposition~\ref{prop:stab} above.
\end{proof}

\section{Constructing the formal degeneration}
\label{sec:gluing}
We outline how the construction of the inverse system of rings in the last two sections allows one to construct a flat deformation by deforming each ring in turn. This section is a variation on Section 6.2.6 in \cite{TropGeom}.

\subsection{Notation}
We define an open set $U^k_\omega$ for each stratum $\omega$, as follows. The sets $U^k_\omega$ together cover the $k$th-order smoothing, and $U^k_\omega$ defines a smoothing of the chart $V(\omega)$ on the central fiber defined in Section~\ref{sec:log_structures}.

\begin{definition}
Let
\[
R^k_\omega := {\varprojlim}_{\omega \subseteq \tau}R^k_{\omega, \tau, \mathfrak{u}_\tau}
\]
and set $U^k_\omega := \Spec R^k_\omega$.
\end{definition}

Since the change of chamber maps are isomorphisms, a different choice of $\mathfrak{u}_\tau$ will yield an isomorphic inverse system -- as proved in 6.2.6 of \cite{TropGeom}.  The main result of this section is:

\begin{prop}
$U^k_\omega$ is a flat deformation of $U^0_\omega$ over $S_k := \Spec k[t]/\big(t^{k+1}\big)$.
\end{prop}

We first compute the central fibre of this degeneration:
\begin{lem}
$U^0_\omega$ is $\Spec k\left[P_{\phi,x}\right]/(t)$ for $x \in \Int(\omega) \cap B_0$.
\end{lem}
\begin{proof}
We give a brief outline of the proof from Lemma 6.30 of \cite{TropGeom}:
\begin{enumerate}
\item As all scattering diagrams are trivial we assume that chambers coincide with maximal cells of $\mathscr{P}$.
\item There are no non-trivial change of chamber maps since the only non-zero elements of $R^0_{\omega,\tau,\sigma}$ for one-dimensional $\tau$ are parallel to $\tau$.
\item Thus the inverse system is just the one made up of all the canonical change of strata maps, and so we recover the toric picture as if there were no scattering.
\end{enumerate}
\end{proof}

The proof of flatness of $U^k_\omega$ over $S_k$ is divided into three parts of increasing complexity, depending on the dimension of the stratum $\omega$.

\subsection{Codimension 0}
For $U^k_\omega$ with $\omega$ two-dimensional we necessarily have that $\sigma_{\mathfrak{u}_\omega} = \omega$. Thus $P_{\phi,\omega,\sigma} = \Lambda_x \times \N$ and $U^k_\omega = U^0_k \times S_k$, i.e.~a trivial deformation.

\subsection{Codimension 1}
For $U^k_\omega$ with $\omega$ one-dimensional we compute an explicit fiber product and show that this is flat.  Following \cite{TropGeom, GS1, Invitation} we fix a one-dimensional $\omega$ and let $\sigma_{\pm}$ be the maximal cells containing $\omega$.  We assume that the piecewise linear function $\phi$ has slope zero on $\sigma_-$ and slope $l\breve{d}_\omega$ on $\sigma_+$; here $\breve{d}_\omega$ is primitive.

There are three rings over which we shall compute the fiber product: $R_{\pm} = R^k_{\omega, \sigma_{\pm},\mathfrak{u}_{\sigma_{\pm}}}$ and $R_\cap = R^k_{\omega \omega \mathfrak{u}_{\sigma_+}}$ - observe the choice of $\sigma_+$ made in defining $R_{\cap}$.  We now define:
\[
f_\omega := f_{\omega,x} \prod_{(\d,x)}{f_{\d,x}}
\]
and regard this as lying in $k[\Lambda_\omega][t]$.  Lemma 6.33 of \cite{TropGeom} then implies that:

\begin{lem}\label{lem:cod1}
The fiber product $R_- \times_{R_{\cap}} R_+$ is isomorphic to the ring
\[
R_{\cup} = k\left[\Lambda_\omega\right][U,V,t]/\left(UV-f_\omega t^l,t^{k+1}\right)
\]
\end{lem}
\begin{proof}
The reader is referred to the proof of Lemma 6.33 of \cite{TropGeom}
\end{proof}

\begin{example}
Consider the local models obtained by the above procedure when $\Delta \cap \rho$ is:
\begin{enumerate}
\item one point with length 2 monodromy polytope;
\item two distinct points, each with simple monodromy.
\end{enumerate}
\noindent Applying Lemma~\ref{lem:cod1} the two cases give the following rings:
\begin{enumerate}
\item $\C[U,V,W,t]/\left(UV - t(W-a)^2, t^k\right)$
\item $\C[U,V,W,t]/\left(UV - t(W-b)(W-c), t^k\right)$
\end{enumerate}
where $a,b,c$ are parameters.

We now consider the singularities of the generic fiber of each of these families.  The first of these exhibits an ordinary double point at $(0,0,a,t)  \in \mathbf{A}_{U,V,W}^3\times\{t\}$, while the second ring gives a smooth affine variety.  We then see the connection between a family of affine varieties defined by varying the parameters $b,c$ and sliding two singularites of an affine structure until they coalesce.  This is precisely the behavour prohibited in \cite{TropGeom, GS1} by demanding the affine manifold be \emph{locally rigid}.
\end{example}

\subsection{Codimension 2 strata}

In \cite{TropGeom, GS1} this is by far the most difficult step.  However working with a more complicated singular locus than used in \cite{TropGeom} does not change this argument and so details of the proof are not recalled here.

As usual, the rings corresponding to the local patch at the zero-cell $\omega$ are given by the inverse limit:
\[
R^k_\omega = {\varprojlim} R^k_{\omega, \tau, \mathfrak{u}_\tau}
\]
The inverse limit is over strata $\tau \supseteq \omega$, with a choice of chamber $\mathfrak{u}_\tau$ for each stratum.  In \cite{TropGeom} it is shown that the choice of this chamber does not change the isomorphism class of the inverse limit.

\section{Local models at vertices}
\label{sec:local_models}

We wish to lift the operation of exchanging corners for singularities described in Section~\ref{sec:affine_manifolds} to a deformation of the rings we have attached to these corners in Sections~\ref{sec:structures},~\ref{sec:scattering}.  To define this deformation we will use an explicit description of the rings at the corners of $B$.  In fact we give two descriptions; the first based on gluing the rings $R^k_{\omega, \tau,\mathfrak{u}}$, the second on the canonical cover construction for surface singularities.  The equivalence of these formulations makes evident that we are constructing $\Q$-Gorenstein deformations.

\subsection{Local description of the affine manifold}

Fix a vertex $\omega$ of $\mathscr{P}$ contained in $\partial B$ and a chart $U \subseteq B$ containing $\omega$ which intersects a minimal number of strata of $\mathscr{P}$.  We shall assume for the rest of this section that:
\begin{enumerate}
\item $\mathscr{P}$ divides $U$ into two regions, described by intersecting $U$ with a pair of 2-cells $\sigma_1$,~$\sigma_2$ which meet along a 1-cell $\tau$.  
\item\label{item:rays} we have fixed a structure $\mathscr{S}$ on $B$. Let $\mathscr{S}_\omega$ be the set of rays in $\mathscr{S}$ intersecting $\omega$.
\item If $\d \in \mathscr{S}_\omega$ then $\d|_{U}$ is supported on $\tau$.
\end{enumerate}

\begin{remark}
These assumptions are automatically satisfied if $B$ is of polygon type.  Also, point~\ref{item:rays} implies that there are two distinguished chambers independent of $k$ whose boundary contains $\tau \cap U$.  We refer to these as $\mathfrak{u}_1$ and $\mathfrak{u}_2$ respectively, where we have suppressed the dependence on $k$.
\end{remark}

\noindent For ease of exposition we will assume without loss of generality that $\phi$ vanishes on the left-hand cone, i.e. on $\mathfrak{u}_1$.

\begin{notation}\ 
\begin{enumerate}
\item Each $\sigma_i$ for $i = 1$,~$2$ contains a 1-cell in $\partial B$ intersecting $\omega$. We denote these 1-cells $\tau_1$,~$\tau_2$ respectively.
\item Let $n_0$ be the unique primitive vector in $\Lambda^\star_\omega$ which annihilates the subspace defined by $\tau$ and evaluates postively on $\mathfrak{u}_1$.
\item Denote by $n_1$,~$ n_2$ the unique primitive vectors in $\Lambda^\star_\omega$ annihilating $\tau_1$,~$\tau_2$ respectively and evaluating non-negatively along $\tau$.
\item Let $f := f_\tau.\prod_{\d}{f_{\d}}$ where $f_\tau$ is the slab function on $\tau$ and the product is over rays $\d$ supported on $\tau$.
\end{enumerate}
\end{notation}

Now we have fixed this notation we describe the rings $R^k_{\omega, \rho, \mathfrak{u}_i}$ for different choices of $\rho$ and $i$.  Recalling that any such ring is a quotient of $k\left[P_{\omega,\phi}\right]$ we fix a generating set for the monoid $P_{\omega,\phi}$.  After taking the projection $m \mapsto \bar{m}$ the generators are distributed in some fashion across the two subcones:

\begin{center}
\includegraphics*[viewport=150 555 350 655]{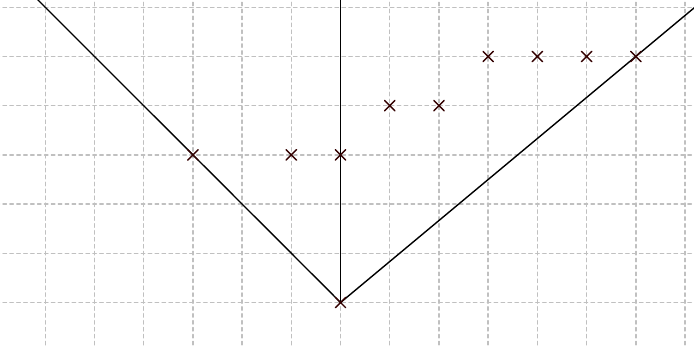}
\end{center}

We will name the generators depending on the cone they project to. $\C\left[P_{\omega,\phi}\right]$ is generated as a $\C[t]$-module by three collections of monomials:
\begin{enumerate}
\item $x_i$ correspond to generators of the left-hand cone (not supported on $\tau$). $x_0$ corresponds to a vector $m_0$ such that $\overline{m}_0 \in \tau_1$.
\item $y_j$ correspond to generators of the right-hand cone (not supported on $\tau$). $y_0$ corresponds to a vector $m_0$ such that $\overline{m}_0 \in \tau_2$.
\item $w$ is the primitive generator of $\tau$.
\end{enumerate}

We recall the standard result in toric geometry that describes the corresponding ideal.
\begin{lem}\label{lem:cls}
If $C$ is a cone in a lattice $M$ with generating set $m_1, \cdots , m_s$ there is a natural short exact sequence:
\[
0 \rightarrow L \rightarrow \Z^{s} \rightarrow M  \rightarrow 0
\]
Writing $l \in L$ via the injective map into $\Z^s$ we can write $l = \sum{l_ie_i}$; now one may form the ideal $I = \left\langle \prod_{l_i>0}{x_i^{l_i}} - \prod_{l_i<0}{x_i^{-l_i}} \right\rangle$, and $k[x_1,\ldots, x_s]/I$ is the affine toric variety $\Spec k[C]$.
\end{lem}
\begin{proof}
See \cite{CLS}, chapter 1.
\end{proof}

The 2-cells $\sigma_1$,~$\sigma_2$ define a pair of cones with their origin at the vertex $\omega$. Let $C_1$,~$C_2$ be the semigroups defined by the integral points of these cones respectively.  Using Lemma~\ref{lem:cls} the relations between the generators specified for the monoid $P_{\omega,\phi}$ are generated by those of the form:
\[
w^\gamma\prod{x_i^{\alpha_i}}\prod{y_j^{\beta_j}} - w^\delta\prod{x_i^{\gamma_i}}\prod{y_j^{\delta_j}}
\]
Recall that in general we have:
\[
R^k_{\omega, \sigma_1, \mathfrak{u}_1} = k\left[P_{\omega,\phi}\right]/I_{\omega,\sigma_1,\sigma_1}
\]
\noindent Now we observe that the order of a monomial $M = t^\gamma\prod{y^{\beta_j}_jw^\alpha}$ in this monoid is given by:
\[
\ord_\tau(M) =  \sum{\beta_j\phi_\omega\left(\bar{m}_j\right)} +\gamma
\]
\noindent This formula, together with the observation that over $\sigma_1$ $\ord_\tau$ is just the $t$-degree fixes an explicit description of the ideal:
\[
I_{\omega,\sigma_1,\sigma_1} = \left\langle M : \ord_{\tau}(M) > k \right\rangle
\]

\begin{remark}
We may view the ring $R^k_{\omega,\sigma_1,\mathfrak{u}_1}$ as a module over $S_k[w]$; letting $S_k[C_1]$, respectively $S_k[C_2]$ be the submodule of $k\left[P_{\omega,\phi}\right]$ generated by the $x_i$ (respectively by the $y_j$) $R^k_{\omega,\sigma_1,\mathfrak{u}_1}$ may be expressed as a pushout:

\begin{displaymath}
    \xymatrix{
    S_k[w] \ar[r] \ar[d] & S_2 \ar[d] \\
    S_1 \ar[r] & R^k_{\omega,\sigma_1,\mathfrak{u}_1} \\}
\end{displaymath}
in which $S_1 = S_k[C_1]/(t^{k+1})$ and
\[
S_2 =  S_k[C_2]/\left\langle \prod{t^\gamma y_j^{\beta_j}} : \sum{\beta_j\phi_\omega\left(\bar{m}_j\right)+\gamma} > k \right\rangle
\]
\end{remark}
\begin{definition}
For each cone $C_i$,~$i = 1,2$, let $C^\circ_i$ be the cone generated by $x_0, \cdots, x_N$, $y_0, \cdots, y_M$ respectively.

\end{definition}

\begin{lem}\label{lem:splitting}
The $S_k[w]$-module
\[
\widetilde{R}^k := S_k[C_1 \backslash \langle w\rangle] \oplus S_k[C_2 \backslash \langle w\rangle] \oplus S_k[w]
\]
is a finitely generated $S_k[w]$-module and there is a surjective homomorphism $\widetilde{R}^k \rightarrow R^k_{\omega,\sigma_1,\mathfrak{u}_1}$.
\end{lem}
\begin{proof}
Observe that the rings $S_k[C_i \backslash \langle w\rangle]$ are finitely generated $S_k[w]$-modules since there are canonical surjective homomorphisms: $S_k[w][C^\circ_i] \rightarrow S_k[C_i \backslash \langle w\rangle]$ for $i=1,2$.  Each factor of $\widetilde{R}^k$ has a canonical map to a term of the push-out diagram above, together defining a map to $R^k_{\omega,\sigma_1,\mathfrak{u}_1}$.  Using this push-out and fixing an element of $R^k_{\omega,\sigma_1,\mathfrak{u}_1}$ it may be expressed as a pair $(u_1,u_2)$; in which $u_i$ is a sum of monomials from $\sigma_i$ for $i=1,2$.  After removing terms involving only the variable $w$ from each $u_i$ we may express any element of $R^k_{\omega,\sigma_1,\mathfrak{u}_1}$ as a triple of the form required.
\end{proof}

We remark that analogous observations may be made about the rings $R^k_{\omega, \sigma_2,\mathfrak{u}_2}$ and $R^k_{\omega, \tau,\mathfrak{u}_1}$. Using this notation we now describe the co-ordinate ring of the affine patch containing the given vertex, that is the inverse limit of the following system.

\begin{displaymath}
    \xymatrix{
    & & R^k_{\Pi} \ar@{.>}[dl] \ar@{.>}[dr] \ar @{.>} @/_2pc/[ddll] \ar @{.>} @/^2pc/[ddrr] \ar@{.>}[dd] \ar@{.>} @/^2pc/[ddd] & & \\
    & R^k_{\omega, \sigma_1,\mathfrak{u}_1} \ar[dl] \ar[dr] & & R^k_{\omega,\sigma_2,\mathfrak{u}_2} \ar[dl] \ar[dr] & \\
    R^k_{\omega, \tau_1,\mathfrak{u}_1} \ar[drr]  & & R^k_{\omega, \tau,\mathfrak{u}_1} \ar[d] &  & R^k_{\omega, \tau_2,\mathfrak{u}_2} \ar[dll] \\
       & & R^k_{\omega,\omega,\mathfrak{u}_1} & & }
\end{displaymath}

\begin{remark}
The inverse limit described above is manifestly isomorphic to the fiber product:
\begin{displaymath}
\xymatrix{
 R^k_{\Pi} \ar@{.>}[r] \ar@{.>}[d] & R^k_{\omega, \sigma_1,\mathfrak{u}_1} \ar[d] \\
  R^k_{\omega,\sigma_2,\mathfrak{u}_2} \ar[r] &  R^k_{\omega, \tau,\mathfrak{u}_1}  }
\end{displaymath}
\end{remark}

If $u \in R^k_{\Pi}$, $u = \left(u_1,u_2\right)$ and the restrictions of $u_i$ to $R^k_{\omega,\tau,\mathfrak{u}_i}$ for $i = 1,2$ respectively are related by the change of chamber map.  Formally, we take the change of strata maps and compose the second with the change of chamber map:
\begin{displaymath}
\xymatrix{
 R^k_{\omega,\sigma_1,\mathfrak{u}_1} \ar[d] & R^k_{\omega, \sigma_2,\mathfrak{u}_2} \ar[d] \\
  R^k_{\omega,\tau,\mathfrak{u}_1} &  R^k_{\omega, \tau,\mathfrak{u}_2} \ar@{-->}[l]^{\theta_{\mathfrak{u}_2, \mathfrak{u}_1}}}
\end{displaymath}

Recall the following facts:
\begin{enumerate}
\item Applying the change of chamber isomorphism $\theta_{\mathfrak{u}_2, \mathfrak{u}_1}$ to variables $x_i$, we have that: $\theta_{\mathfrak{u}_2, \mathfrak{u}_1}\left(x_i\right) = f^{\left\langle n_0,\bar{m}\right\rangle}x_i$. 
\item  There is a similar formula for the $\theta_{\mathfrak{u}_2, \mathfrak{u}_1}(y_j)$ and $w$ is always mapped to itself, as $n_0$ annihilates the tangent space to $\tau$.
\item The rings $R^k_{\omega,\tau,\mathfrak{u}_i}$,~$i=1,2$ have been localised at the slab function, ensuring that change of chamber map is an isomorphism.
\end{enumerate}

We are now in a position to give an elementary description of the formal smoothing of the affine chart at a boundary vertex obtained from the Gross--Siebert reconstruction algorithm.
\begin{definition}\label{def:icup}
$R^k_{\cup} = S_k[X_i, Y_j, W, t : 0 \leq i \leq N, 0 \leq j \leq M ]/I_{\cup}$.
To define $I_\cup$ consider each binomial relation
\[
w^{\eta_1}\prod_{i,j}{x^{\alpha_i}_iy^{\beta_j}_j} = t^\chi w^{\eta_2}\prod_{k,l}{x^{\gamma_k}_ky^{\delta_l}_l}
\]
in the usual monoid over $\phi$ on $C_1 \cup C_2$. We define an element of $I_{\cup}$ which may take one of two forms; if the monomials correspond to a lattice vector in $C_1$ consider the polynomial
\[
f^{-\sum_l{\delta_l\left\langle n_0, m_l\right\rangle}}W^{\eta_1}\prod_{i,j}{X^{\alpha_i}_iY^{\beta_j}_j} - f^{-\sum_j{\beta_j\left\langle n_0, m_j\right\rangle}}t^\chi W^{\eta_2}\prod_{k,l}{X^{\gamma_k}_kY^{\delta_l}_l}
\]
otherwise, if it is over $C_2$, consider the polynomial
\[
f^{\sum_k{\gamma_k\left\langle n_0, m_k\right\rangle}}W^{\eta_1}\prod_{i,j}{X^{\alpha_i}_iY^{\beta_j}_j} - f^{\sum_i{\alpha_i\left\langle n_0, m_i\right\rangle}}t^\chi W^{\eta_2}\prod_{k,l}{X^{\gamma_k}_kY^{\delta_l}_l}
\]
Here $f$ is considered as an element of $S_k\left[W\right]$ (rather than $S_k\left[w\right]$).  Divide out the given polynomial by as many factors of $f$ as possible and append it to the generating set of $I_\cup$.  For clarity we shall suppress the $W^{\eta_i}$ in these relations from now on.
\end{definition}

\begin{prop}\label{prop:ring_iso}
There is a ring isomorphism $\Phi\colon R^k_{\cup} \rightarrow R^k_{\Pi}$ given on generators by:
\begin{align*}
 X_i &\mapsto \left(x_i,f^{\left\langle n_0, m_i \right\rangle}x_i\right) \\
 Y_j &\mapsto \left(f^{\left\langle n_0, m_j \right\rangle}y_j,y_j\right) \\
 W &\mapsto \left(w,w\right) \\
 t &\mapsto \left(t,t\right)
\end{align*}
\end{prop}
\begin{remark}
Compare with the description around an interior 1-cell given in \cite{TropGeom}.  These rings are more complicated but the change of chamber map in the fiber product is essentially the same. 
\end{remark}
\begin{proof}

To show this map is well-defined we consider the images under $\Phi$ of the generators of $I_\cup$. Indeed, we may simply compute $\Phi$:
{\footnotesize
\begin{align*}
 & \Phi\left(f^{\sum_k{\gamma_k\left\langle n_0, m_k\right\rangle}}\prod_{i,j}{X^{\alpha_i}_iY^{\beta_j}_j} \right) = \\
 = & f^{\sum_k{\gamma_k\left\langle n_0, m_k\right\rangle}}\left(f^{-\sum{\beta_j\left\langle n_0,m_j\right\rangle}}\prod_{i,j}x^{\alpha_i}_iy^{\beta_j}_j,f^{\sum{\alpha_i\left\langle n_0,m_i\right\rangle}}\prod_{i,j}{x^{\alpha_i}_iy^{\beta_j}_j}\right) \\ 
 = & f^{\sum_k{\gamma_k\left\langle n_0, m_k\right\rangle} + \sum_i{\alpha_i\left\langle n_0, m_i\right\rangle}}\left(f^{-\sum{\beta_j\left\langle n_0,m_j\right\rangle} - \sum{\alpha_i\left\langle n_0,m_i\right\rangle}}\prod_{i,j}x^{\alpha_i}_iy^{\beta_j}_j,\prod_{i,j}{x^{\alpha_i}_iy^{\beta_j}_j}\right) \\
 = &  f^{\sum_k{\gamma_k\left\langle n_0, m_k\right\rangle} + \sum_i{\alpha_i\left\langle n_0, m_i\right\rangle}}\left(f^{-\sum{\delta_l\left\langle n_0,m_l\right\rangle} - \sum{\gamma_k\left\langle n_0,m_k\right\rangle}}\prod_{k,l}{x^{\gamma_k}_ky^{\delta_l}_l},\prod_{k,l}{x^{\gamma_k}_ky^{\delta_l}_l}\right) \\
 = & f^{\sum_i{\alpha_i\left\langle n_0, m_i\right\rangle}}\left(f^{-\sum{\delta_l\left\langle n_0,m_l\right\rangle}}\prod_{k,l}{x^{\gamma_k}_ky^{\delta_l}_l},f^{\sum_k{\gamma_k\left\langle n_0, m_k\right\rangle}}\prod_{k,l}{x^{\gamma_k}_ky^{\delta_l}_l}\right) \\
 = & \Phi\left(f^{\sum_i{\alpha_i\left\langle n_0, m_i\right\rangle}}\prod_{k,l}{X^{\gamma_k}_kY^{\delta_l}_l} \right)
\end{align*}}

To show $\Phi$ is surjective we use Lemma~\ref{lem:splitting}, which gives a generating set for the algebras $R^k_{\omega,\sigma_i,\mathfrak{u}_i}$ as $S_k\left[w\right]$ modules.  Fix an element $\left(u_1,u_2\right) \in R^k_\Pi$, without loss of generality we assume that there are no terms in $u_i, i=1,2$ involving only $w$ as any polynomial $g(w)$ may be accounted for by taking $\Phi(g(W))$.  Now we (non-uniquely) write $u_1 = \sum_k{c_k\prod_i{x^{\alpha_{i,k}}_i}} + h_1\left(y_j : 0 \leq j \leq M \right)$ where the coefficents $c_m$ lie in the ring $S_k[w]$.  Similarly we write $u_2 = \sum_l{c_l\prod_j{y^{\beta_{j,l}}_j}} + h_2\left(x_i : 0 \leq i \leq N \right)$ using the same coefficent ring.

We claim that the pair $(u_1,u_2)$ is in $R^k_{\Pi}$ if and only if it is equal to:
\[
\Phi\left(\sum_k{c_k\prod_i{x^{\alpha_{i,k}}_i}} + \sum_l{c_l\prod_j{y^{\beta_{j,l}}_j}}\right)
\]
By the previous calculation this is certainly in the fiber product; furthermore this element agrees with all the $x_i$ terms in $f_1$ and the $y_j$ terms in $f_2$ by definition.  All that remains is to check that this uniquely determines the $h_1$ and $h_2$. However the change of strata map is the identity on $h_1$ and $h_2$ and so we may express these in terms of previously determined quantities, for example:
\[
h_1 = \theta_{\mathfrak{u_2},\mathfrak{u_1}}\psi_{\left(\omega,\sigma_2\right),\left(\omega,\tau\right)}\left(\sum_l{c_l\prod_j{y^{\beta_{j,l}}_j}}\right)
\]
We next show that this map is injective.  Assume we have a element $u \in R^k_\cup$ that is mapped to a pair $(u_1,u_2)$ such that $u_1 \in I^k_{\omega,\sigma_1,\mathfrak{u}_1}$ and $u_2 \in I^k_{\omega,\sigma_2,\mathfrak{u}_2}$.  Observe that we may rewrite any monomial $\prod_{i,j}x_i^{\alpha_i}y_j^{\beta_j}$, using the toric relations, in one of the following two forms:
\begin{enumerate}
\item $\prod_{i,j}x_i^{\alpha_i}y_j^{\beta_j} - \prod_k{x_k^{\gamma_k}}$
\item $\prod_{i,j}x_i^{\alpha_i}y_j^{\beta_j} - \prod_l{y_l^{\delta_l}}$
\end{enumerate}
From Definition~\ref{def:icup} we have a relations in $I_\cup$ of the form:
\begin{enumerate}
\item $\prod_{i,j}X_i^{\alpha_i}Y_j^{\beta_j} - f^{\sum{\alpha_i\left\langle n_0, m_i\right\rangle}}\prod_k{X_k^{\gamma_k}}$
\item $\prod_{i,j}X_i^{\alpha_i}Y_j^{\beta_j} - f^{-\sum{\beta_j\left\langle n_0, m_j\right\rangle}}\prod_l{Y_l^{\delta_l}}$
\end{enumerate}
Thus we can assume that there are no terms involving both the $X_i$ and the $Y_j$ appearing in a representative of $R^k_\cup$, but by Proposition~\ref{prop:ring_iso} $\Phi$ is the identity onto one of the two factors.  Since the image onto this factor is in $I^k_{\omega,\sigma_i,\mathfrak{u}_i}$ for some $i$ we may infer that the original element is in $I_\cup$.
\end{proof}

\subsection{Orbifolding the local models}
We conclude this section by exhibiting a construction of the canonical cover for these rings; this will be used in the next section to construct a $\Q$-Gorenstein deformation.

Given a vertex $v \in B$ fix a chart of $B$ containing $v$ and let $C$ denote the tangent cone at $v$. We shall assume for the remainder of this section that
\begin{enumerate}
\item $\mathscr{P}$ splits $C$ into two cones $C_i$~,$i=1,2$, divided by a ray $L$.
\item Denoting the primitive generators of $C$ by $v_1$,~$v_2$ respectively we have that $v_1+v_2 \in L$.
\end{enumerate}

\begin{lem}\label{lem:local-polygons}
Given a Fano polygon $P$ fix a vertex $v$, its tangent wedge $C$ and the ray $L$ of the spanning fan of $Q$ meeting $v$. The pair $(C,L)$ satisfies the two conditions above. 
\end{lem}
\begin{proof}
The first condition is obvious, the spanning fan introduces precisely one new ray intersecting $v$. For the second condition note that an edge of $P$ may be put into the following standard form:
\begin{center}
\includegraphics*[viewport=147 552 293 668]{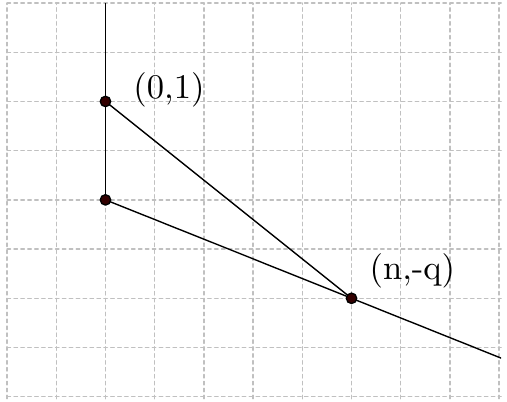}
\end{center}
with the vertices of $P$ at $(0,1), (n,-q)$.  Taking the dual cone:
\begin{center}
\includegraphics*[viewport=162 553 275 669]{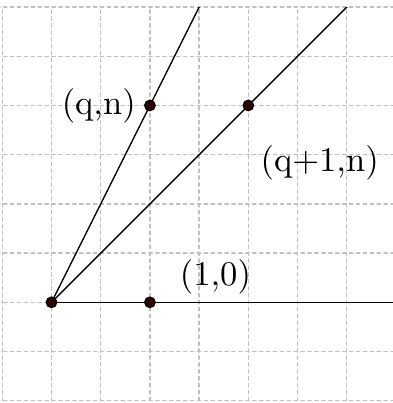}
\end{center}
We see the (rational) generators of this cone are $(1,0),(q,n)$, the ray $L$ defined by the normal to the edge of $P$ is generated by $(q+1,n)$ and satisfies the second condition.
\end{proof}

We recall the canonical cover construction for the singularity $X = \frac{1}{n}(1,q)$, for which we use the following notation:

\begin{notation}\ 
\begin{enumerate}
\item Define $p:=q+1$.
\item Let $w := \text{hcf}(n,p)$ and define $a,r$ by requiring that $n = wr, p = wa$, so in particular $q = wa-1$.
\item Define $m$,~$w_0$ by $w = mr + w_0$ with $0 \leq w_0 < r$.
\end{enumerate}
\end{notation}

\begin{remark}
The singularity content of the singularity $X = \frac{1}{n}(1,q)$ is precisely $m$.
\end{remark}

Having fixed this notation the canonical cover of $X$ is:
\begin{construction}\label{cons:loc-qgor}
Letting $X = \frac{1}{n}(1,q)$ there is an embedding $X \hookrightarrow \frac{1}{r}(1,q,a)$ which takes $X$ onto the hypersurface $\{xy = z^w\}/\mu_r$. The $\Q$-Gorenstein deformations of $X$ are determined by considering the space $\C^{m+1}$ of degree-$m$ polynomials $f_m$ and forming the family of hypersurfaces
\[
\{xy = z^{w_0}f_m(z^r)\}/\mu_r
\]
We shall show that our local model $R_\cup$ is always of this form and thus that the space of polynomials defined by the log-structure on this line segment may be identified with the parameter space of $\Q$-Gorenstein deformations.
\end{construction}

In order to prove this relation, we compare the cones constructed in the proof of Lemma~\ref{lem:local-polygons} to Construction~\ref{cons:loc-qgor}.

\begin{construction}\label{cons:monoids}
Given $X = \frac{1}{n}(1,q)$, the fan of $X$ is given by
\[
\Cone((0,1),(n,-q))
\]
as in the proof of Lemma~\ref{lem:local-polygons}.  This is isomorphic to the cone $\Cone\left((1,0),(0,1)\right)$ in the lattice: $\Z^2 + \frac{1}{n}(1,q) $.  Similarly $Y = \frac{1}{r}(1,q,a)$ is determined by $\Cone\left((1,0,0),(0,1,0),(0,0,1)\right)$ in the lattice: $\Z^3 + \frac{1}{r}(1,q,a)$. Following Construction~\ref{cons:loc-qgor} we should consider the hypersurface $\{xy = z^w\}$.  This is the image of the embedding $X \hookrightarrow Y$. This embedding is induced by a map $\iota\colon \Z^2 + \frac{1}{n}(1,q) \rightarrow \Z^3 + \frac{1}{r}(1,q,a)$ between the respective lattices which may be expressed as the following matrix, which we also call $\iota$.
\[
\iota = \begin{pmatrix}
w & 0 \\
0 & w \\
1 & 1
\end{pmatrix}
\]
In particular $\iota(\frac{1}{n}(1,q)) = \frac{1}{n}(w,qw,1+q) = \frac{1}{wr}(w,qw,wa) = \frac{1}{r}(1,q,a)$.  We wish to compute the map between the dual lattices induced by $\iota$. Observe that $(\Z^2+\frac{1}{n}(1,q))^\vee$ is the sublattice
\[
\left\{\alpha \in {\Z^2}^\vee : \alpha((1,q)) \in n\Z  \right\}
\]
of the dual lattice $\Z^2{}^\vee$.  There is an analogous expression for the lattice dual to $\Z^3 + \frac{1}{r}(1,q,a)$.  From the matrix $\iota$ we may easily compute $\iota^\star$, in particular $\iota^\star(x^r) = x^n$,~$\iota^\star(y^r) = y^n$ and $\iota^\star(z^r) = x^ry^r$.
\end{construction}
\begin{remark}
Recall that the image of $\iota^\star$ is a sublattice of $\Z^2{}^\vee$.  The lattice elements corresponding to $x^n, x^ry^r, y^n$ are all primitive in this lattice, for example $x^ry^r$ is the generator of the cone previously called $W$.
\end{remark}

Using these constructions we shall define a ring $R'{}^k_\cup$ and prove that it is isomorphic to $R^k_\cup$.
\begin{definition}
Given a zero stratum $v$ of $\mathscr{P}$ contained in $\partial B$ we may form the pair $(C,L)$ as above.  Note that $C$ need not be strictly convex.  In particular we may define the integers $n,q,w$ for this cone.
\[
R'{}^k_\cup = S_k[x,y,z]^{\mu_r}/\big(xy = t^lz^{w_0}f_\tau(z^r)\big)
\]
where the $\mu_r$ action has weights $(1,q,a)$ and $l$ is the slope of the piecewise linear function $\phi$.
\end{definition}

\begin{prop}
\label{prop:orbifold-model}
$R'{}^k_\cup$ is isomorphic to $R^k_\cup$.
\end{prop}
\begin{proof}
There is an obvious spanning set of $R'{}^k_\cup$ as an $S_k$-module; namely monomials with exponents in the sublattice of $\Z^3{}^\vee$ dual to $(\Z^3 + \frac{1}{r}(1,q,a))$.  Consider the submodule generated by the monomials $x^az^b$ and $y^cz^d$; these give a basis for $R'{}^k_\cup$ as a $S_k$-module. Making the analogous statement for $R^k_\cup$ we observe that $R^k_\cup$ is generated as an $S_k$-module by monomials with exponents projecting to integral points in the cone $C$.  There is an obvious identification of these two bases, which extends linearly to a map of $S_k$-modules; we now show this is an isomorphism of algebras. As a preliminary step we replace $f_\tau$ in the definition of $R'{}^k_\cup$ with $f = f_\tau\prod f_\d$ where the product is over the rays $\d$ of the scattering diagram supported on $\tau$. Note each $f_\d$ is invertible in $R'{}^k_\cup$, so an automorphism of $S_k[x,y,z]^{\mu_r}$ sending $xy \mapsto xy\prod f_\d$ induces an isomorphism of $R'{}^k_\cup$ with $S_k[x,y,z]^{\mu_r}/\big(xy = t^lz^{w_0}f(z^r)\big)$.

Fix $U, V \in R^k_\cup$ and write $U = \bar{U}t^{l_1}$ and $V = \bar{V}t^{l_2}$ where $\bar{U} \in C_1$ and $\bar{V} \in C_2$. Now take the corresponding elements in $R'{}^k_\cup$: $\iota^\star(x^az^bt^{l_1}), \iota^\star(y^cz^dt^{l_2})$.  Suppose we have that $UV$ projects to an element in $C_1$ and write $-\langle n_0, \bar{V}\rangle = \gamma$ so that $UV = \prod{X^{a_i}_i}W^bt^{l_1+l_2+\gamma l}f^\gamma$ where the $X_i$ correspond to elements of the Hilbert basis of $C_1$.  Writing

\[
\iota^\star(x^az^bt^{l_1}).\iota^\star(y^cz^dt^{l_2}) = \iota^\star(x^ay^cz^{b+d}t^{l_1+l_2})
\]
$UV$ in $C_1$ means that $c < a$ so using the relations in $R'{}^k_\cup$,
\[
\iota^\star(x^ay^cz^{b+d}t^{l_1+l_2}) =  \iota^\star\left(x^{a-c}z^{b+d+c.w_0}t^{l_1+l_2+cl}f(z^r)^c\right)
\]
\noindent Our $S_k$-module isomorphism identifies
\[
\iota^\star\left(x^{a-c}z^{b+d+c.w_0}t^{l_1+l_2+cl}\left(f(z^r)\right)^c\right)
\]
\noindent with $\prod{X^{a_i}_i}W^bt^{l_1+l_2+cl}f^c$; thus we only need to show that $\gamma = c$. Recall we have identifed $C$ with the quadrant in a sublattice of $\Z^2{}^\vee$. Therefore we can compute $\langle n_0, (v_1,v_2)\rangle$ directly. The primitive generator of $L$ in this sublattice of $\Z^2{}^\vee$ is $(r,r)$; the obvious element annihilating $(r,r)$ is $(1,-1)$, but this has index $w$ ($(1,1) = wr\frac{1}{n}(1,q) - wa(0,1)$), so in fact $\langle n_0, (v_1,v_2)\rangle = (v_1-v_2)/w$. Now consider an element $\iota^\star{y^cz^d} = x^{wc+d}y^{d}$, evaluating $\gamma = \langle n_0, (v_1,v_2)\rangle$ for this lattice point we find that indeed $\gamma = c$.
\end{proof}

\section{Smoothing quotient singularities of del-Pezzo surfaces}

\label{sec:smoothing}

Consider an affine manifold of polygon type, $B_Q$.  In the previous sections we have:
\begin{enumerate}
\item Defined the notion of a \emph{one-parameter degeneration} of such affine manifolds
\item Defined a family of log structures on the variety $X_0(B_Q, \mathscr{P},s)$
\item Outlined the Gross--Siebert algorithm for constructing a formal smoothing of this using the log-structure
\item Explicitly computed the various rings and the family in the case of an isolated boundary singularity.
\end{enumerate}

In this section we combine these to construct a flat family $\mathcal{X} \rightarrow \Spec\C[\alpha]\llbracket t \rrbracket$ which will satisfy the conditions of Theorem~\ref{thm:families}, namely:
\begin{itemize}
\item Fixing a nonzero $\alpha$ the restriction of $\cX$ over $\Spec \C\llbracket t\rrbracket$ is the flat formal family produced by the Gross--Siebert algorithm.
\item Fixing $\alpha = 0$ the restriction of $\cX$ over $\Spec \C\llbracket t\rrbracket$ is precisely the restriction of the Mumford degeneration of the pair $(Q,\mathscr{P})$.
\item Fixing $t=0$, the restriction of $\cX$ is $X_0(Q, \mathscr{P}, s) \times \Spec \C[\alpha]$.
\item For each boundary zero-stratum $p$ of $X_0(Q, \mathscr{P}, s)$ there is neighbourhood $U_p$ in $\cX$ isomorphic to a family $\mathcal{Y}\rightarrow \Spec\C[\alpha]$ obtained by first taking a one-parameter $\Q$-Gorenstein smoothing of the singularity of $X_Q$ at $v$, taking a simultaneous maximal degeneration of every fiber and restricting to a formal neighbourhood of the central fiber.
\end{itemize}

The obstacle to simply applying the Gross--Siebert algorithm to the family fiberwise is the jump in the log-structure at the central fiber; sections defining the log-structure are not permitted to vanish on any zero stratum.  In fact we wish to choose log-structures from a different bundle at the central fiber, as the singular locus has changed. Therefore we have no \emph{a priori} reason to suppose these glue to a family.  However, we shall prove that our explicit construction at boundary zero-strata enables one to extend the obvious family over $\C^\star$ to one over $\C$.

Recall we have a family of affine manifolds $\pi_Q\colon \mathcal{B}_Q \rightarrow \R$ defined by smoothing the corners, as described in Section~\ref{sec:affine_manifolds}.  Fix a one parameter family of log-structures compatible with the family of affine manifolds in the sense of Definition~\ref{def:compatible}.

\begin{remark}
Consider the scattering diagram $\mathfrak{D}_\omega$ at the central vertex; this is equivalent to a scattering diagram of the following form:
\[
\mathfrak{D} = \left\{\left(\R\bar{m}_i, \prod_{j,k}\left(1+c_{ijk}z^{-m_{ijk}}\right)\right): 1 \leq i \leq p\right\}
\]
Assuming $c_{ijk} \in \C[\alpha]$ the assumptions on a family of log structures imply that $c_{ijk} \in \alpha.\C[\alpha]$.
\end{remark}

\begin{definition}
For this section a \emph{family of scattering diagrams} (with parameter $\alpha$) is a scattering diagram defined via a map $r\colon P \rightarrow M$ and an $\mathfrak{m}$-primary ideal $I$, but now for $\d \in \mathfrak{D}$,~$f_\d \in \C[\alpha][P]/I$.  Further, write $\mathfrak{D}(\alpha)$ for the scattering diagram where all the functions have been evaluated at $\alpha$.
\end{definition}

\begin{lem}\label{lem:family_scattering}
Given a family of scattering diagrams $\mathfrak{D}$ there is another one $S_I\big(\mathfrak{D}\big)$ such that:
\[
S_I\big(\mathfrak{D}\big)(\alpha) = S_I\big(\mathfrak{D}(\alpha)\big)
\]
for all $\alpha \in \C$.
\end{lem}
\begin{proof}
We use the notion of a universal scattering diagram, indeed, writing:
\[
\mathfrak{D} = \left\{\left(\R\bar{m}_i, \prod_{j,k}\left(1+c_{ijk}z^{-m_{ijk}}\right)\right): 1 \leq i \leq p\right\}
\]
we can form:
\[
\mathfrak{D}' = \left\{\left(\R\bar{m}_i, \prod_{j,k}\left(1+t_{ijk}z^{-r\left(m_{ijk}\right)}\right)\right): 1 \leq i \leq p\right\}
\]
Where in the first scattering diagram is $c_{ijk}$ is polynomial in $\alpha$ and the second scattering diagram is defined over the ring $\C[M]\llbracket \left\{t_{ijk}\right\}\rrbracket$. In fact, following \cite{TropGeom}, this scattering diagram is defined over $\C[Q]$ where $Q \subseteq M \oplus \N^l$ is the monoid freely generated by pairs $\left(-r\left(m_{ijk}\right) , e_{ijk} \right)$, where $e_{ijk}$ corresponds to $t_{ijk}$. Thus given an ideal $I$ of $P$ we obtain a scattering diagram $S_{I'}\left(\mathfrak{D}'\right)$ by reduction modulo $I' = \phi^{-1}(I)$ where:
\[
\phi\colon \C[Q] \rightarrow  \C[\alpha][P]
\]
via $t_{ijk}z^{-r\left(m_{ijk}\right)} \mapsto c_{ijk}z^{-m_{ijk}}$.  Composing this with the evaluation map $\psi_{\alpha}\colon \C[\alpha][P] \rightarrow \C[P]$ we obtain a scattering diagram: $\psi_{\alpha}\circ\phi\left(S_{I'\left(\mathfrak{D}'\right)}\right)$, which must be equivalent to $S_I\left(\mathfrak{D}\left(\alpha\right)\right)$ by uniqueness.  Thus we set $S_I\left(\mathfrak{D}\right) = \phi\left(S_{I'}\left(\mathfrak{D}'\right)\right)$.
\end{proof}

\begin{prop}\label{prop:localmodel}
Varying $\alpha$ gives an algebraic family $\pi \colon \Spec \widetilde{R}^k_{\omega} \rightarrow \Spec \C\left[\alpha\right]$.
\end{prop}

\begin{proof}
We construct $\C[\alpha]$-algebras $\widetilde{R}^k_{\omega}$ the fibers of which are the rings $R^k_{\omega}$ defined using the various log structures.

First let $\omega$ be a vertex contained in $\partial B$. From Section~\ref{sec:local_models} we have a description of these rings via the isomorphism with the ring $R^k_\cup$.  We denote by $\widetilde{R}^k_{\cup}$ the $\C[\alpha]$-algebra:
\[
\C[\alpha]\left[X_i,Y_j,W\right]/I_{\cup}
\]
Let $\omega$ be the central vertex of $\mathscr{P}$.  The ring $R^k_{\omega}$ is a fiber product of rings of the form $R^k_{\omega \tau \mathfrak{u}}$ which is a quotient of the algebra $\C\left[P_{\omega, \phi}\right]$.  We form the trivial algebra $\C[\alpha]\left[P_{\omega, \phi}\right]$ and so form the analogous rings $\tilde{R}^k_{\omega, \tau, \mathfrak{u}}$.  Firstly setting
\[
\tilde{R}^k_{\omega, \tau, \sigma_\mathfrak{u}} = R^k_{\omega, \tau, \sigma_\mathfrak{u}}\otimes_{\C} \C[\alpha]
\]
and then defining:
\[
\tilde{R}^k_{\omega, \tau, \mathfrak{u}} = \left(\tilde{R}^k_{\omega, \tau, \sigma_\mathfrak{u}}\right)_{f_\tau}
\]
noting again that $f_\tau$ has non-trivial dependance on $\alpha$.  The change of chamber maps now give morphisms:
\[
\theta_{\mathfrak{u},\mathfrak{u}'} \colon \tilde{R}^k_{\omega, \tau, \mathfrak{u}} \rightarrow \tilde{R}^k_{\omega, \tau, \mathfrak{u}'}
\]
via the natural extension of the original definition:
\[
\theta_{\mathfrak{u},\mathfrak{u}'}\left(z^m\right) = \left(\prod{f_\d}\right)^{\left\langle n_0, m\right\rangle }z^m
\]
These are isomorphisms of the rings $\tilde{R}^k_{\omega, \tau, \mathfrak{u}}$, giving $\tilde{R}^k_\omega$ the structure of a $\C[\alpha]$-algebra by taking the inverse limit of the rings $\tilde{R}^k_{\omega, \tau, \mathfrak{u}}$. Finally we need to check that varying $\alpha$ the functions on rays of the scattering diagram are polynomial in $\alpha$, but this we know from Lemma~\ref{lem:family_scattering}.
\end{proof}

\begin{definition}
We define the scheme $\cX_Q \rightarrow \Spec \C[\alpha]\llbracket t \rrbracket$ via the inverse limit over the system $\tilde{R}^k_\omega$, each of which is a $\C[\alpha]\llbracket t \rrbracket$-algebra.
\end{definition}

\begin{remark}
In Theorem~\ref{thm:families} we demand that $\cX_Q$ is flat over $\Spec \C[\alpha]\llbracket t \rrbracket$. Since flatness is local, we can consider $\C[\alpha]\llbracket t \rrbracket$-algebras $\tilde{R}^k_\omega$ for each zero-dimensional stratum $\omega$. We break these into two cases:
\begin{itemize}
\item If $\omega$ is a boundary zero-stratum flatness is an immediate consequence of Proposition~\ref{prop:orbifold-model} which gives an explicit description of this algebra.
\item If $\omega$ is the central vertex we observe that by Lemma~\ref{lem:family_scattering} the functions $f_\d$ on each ray of the scattering diagram at order $k$ is an element of $\C[\alpha,t]/(t^{k+1})$. We can now follow the proof of the case $\dim \omega = 0$ in Theorem 6.32 of \cite{TropGeom} over the ring $\C[\alpha,t]/(t^{k+1})$.
\end{itemize}
\end{remark}
We now prove that this satisfies the various conditions of Theorem~\ref{thm:families}, first identifying the restriction to $\alpha = 0$.
\begin{prop}\label{prop:central}
The restriction of $\cX_Q \rightarrow \Spec \C[\alpha]\llbracket t \rrbracket$ to $\alpha = 0$ is a thickening of the central fiber of the Mumford degeneration.
\end{prop}
\begin{proof}
Firstly we address the local model $R^k_\omega$ for $\omega$ the vertex of $\mathscr{P}$ in the interior of $B$.  However the fiber $\alpha = 0$ is trivial, in the sense that all the slab functions are equal to 1, therefore the scattering diagram is trivial and there is a bijection between chambers and 2-cells of $\mathscr{P}$.  Therefore the inverse limit simply reconstructed a local piece of the Mumford degeneration, as claimed.

Of greater interest are the local models at the vertices.  As we remarked we cannot use the inverse limit, but rather we use the $R^k_\cup$ model constructed above. Using the notation from Section~\ref{sec:local_models} we recall that the non-trivial relations were between generators projecting to different cones, for example:
\[
\left(\prod{f_{\d}}\right)^{\sum_k{\gamma_k\left\langle n_0, m_k\right\rangle}}\prod_{i,j}{X^{\alpha_i}_iY^{\beta_j}_j} = \left(\prod{f_{\d}}\right)^{\sum_i{\alpha_i\left\langle n_0, m_i\right\rangle}}\prod_{k,l}{X^{\gamma_k}_kY^{\delta_l}_l}
\]
Observe that $\prod{f_\d} = f_\tau \prod_{\d \textrm{ ray}}{f_\d}$ where $f_\tau$ is the slab function associated to $\tau$, and in particular that the our assumptions on the one-parameter family of log-structures imply that $f_\tau|_{\alpha = 0} = w^{\deg f_\tau}$.  Observe also that $\prod_{\d \textrm{ ray}}{f_\d}|_{\alpha = 0} = 1$.  This is a consequence of the fact that $S\big(\mathfrak{D}\big)(\alpha) = S\big(\mathfrak{D}(\alpha)\big)$: for the scattering diagram at the central vertex, setting $\alpha = 0$ the scattering diagram is trivial -- every line has function $f_\d = 1$.  Therefore this is already consistent to all orders.  The rays of this scattering diagram propagate until they intersect $\partial B$ and indeed give all the rays in this structure.  Combining these two observations we see that the fiber over zero has co-ordinate ring with relation:
\[
\left(w^l\right)^{\sum_k{\gamma_k\left\langle n_0, m_k\right\rangle}}\prod_{i,j}{X^{\alpha_i}_iY^{\beta_j}_j} = \left(w^l\right)^{\sum_i{\alpha_i\left\langle n_0, m_i\right\rangle}}\prod_{k,l}{X^{\gamma_k}_kY^{\delta_l}_l}
\]
Here $l = \deg(f_\tau)$, which is also the lattice length of the monodromy polytope of the discriminant locus on $\tau$.  Thus the local models near the boundary vertices, when $\alpha$ is set equal to zero, recover the local models for the Mumford degeneration.
\end{proof}

To conclude the proof of Theorem~\ref{thm:families} we need to show that near the boundary vertices the family $\mathcal{X}_Q$ is induced by a $\Q$-Gorenstein smoothing of the singularities of $Q$.

\begin{prop}
The family obtained in Proposition~\ref{prop:localmodel} in each of the charts containing a vertex of $Q$ is isomorphic to a one parameter $\Q$-Gorenstein smoothing.
\end{prop}
\begin{proof}
This is immediate from Proposition~\ref{prop:orbifold-model}, as we may rewrite the families using the canonical cover. Indeed, by Proposition~\ref{prop:orbifold-model} deforming the log-structure simply deforms the equation in this cover, so in particular $R'{}^k_\cup$ is defined for any fiber, not just away from the special fiber.
\end{proof}

We remark that for each $k$,~$f = f_\tau\prod_{\d}{f_\d}$ is a polynomial in $\alpha$, but as $k \rightarrow \infty$ the degree of this polynomial will, in general, tend to infinity.  However there are local co-ordinates near boundary vertices with respect to which the family $\mathcal{X}_Q$ is algebraic to all orders.

\section{Ilten families}
\label{sec:ilten}

We have studied Fano polygons $P$ and smoothings of the associated toric varieties $X_P$.  From the perspective of mirror symmetry \cite{Sigma, Fano} Fano polygons have a different interpretation -- as Newton polygons of a Laurent polynomial $W$ referred to as the mirror \emph{superpotential}.  Indeed, information pertaining to the enumerative geometry of a smoothing of $X_P$ is encoded in the periods of $W$.  However, there are potentially infinitely many Laurent polynomials (with different Newton polygons) that encode this enumerative information. These Laurent polynomials are related by certain birational transformations, referred to as \emph{mutations} \cite{Sigma}, or \emph{symplectomorphisms of cluster type}~\cite{oldandnew}. Mutation of $W$ defines an operation on the Newton polygon $P$ of $W$ and, by duality, an operation on $Q = P^\vee$. This dual action is the restriction of a piecewise linear transformation on the lattice $M$, where $Q \subset M_{\R}$.  This piecewise linear transformation is precisely the transition function between the two charts defining the affine manifold obtained by exchanging a corner of $Q$ for an interior singular point, as described in Section~\ref{sec:affine_manifolds}.  One may then consider a family of affine manifolds in which the singularity is introduced, traverses its monodromy invariant line, and creates a corner in the opposing edge.  This is made precise in the following way:

\begin{prop}
Given a mutation between polygons $Q, Q' \subset M_\R$ there is family of affine manifolds $\pi \colon \mathcal{B} \rightarrow [0,1]$ for which:
\begin{enumerate}
\item $Q = \pi^{-1}(0)$, $Q' = \pi^{-1}(1)$.
\item The generic fiber contains a single type-1 singularity.
\end{enumerate}
\end{prop}
This will be referred to as the \emph{tropical Ilten family}.
\begin{proof}
Take $\pi\colon \mathcal{B} \rightarrow [0,1]$ to be the trivial family with fiber $Q$.  Construct a line segment $l$ contained in the interior of $Q$ as follows;  The mutation is defined as a piecewise linear transformation on $Q$ and $Q'$, there is a distinguished line dividing $M$ into two chambers; intersecting this line with $Q$ defines $l$.  We shall refer to the two chambers contained in $Q$ as $Q_1, Q_2$ and $Q'_1, Q'_2$ in $Q'$.  Take a parameterization of $l$, writing now $l\colon [0,1] \rightarrow Q$.

We define the affine structure on the total space by covering it with two charts:
\begin{enumerate}
\item Let $\mathcal{B}$ be the topological space $Q \times [0,1]$.
\item Take $U_1 \subset \mathcal{B}$ to be
\[
U_1 = \mathcal{B} \backslash \{(l(t),u) : \text{$u,t \in [0,1]$,~$t \leq u $ and $ u \neq 0$} \}
\]
\item Similarly take $U_2 \subset \mathcal{B}$ to be 
\[
U_2 = \mathcal{B} \backslash \{(l(t),u) : \text{$u,t \in [0,1], t > u$ and $u \neq 1$} \}
\]
\item Take the transition function such that the fiber $\pi^{-1}(1)$ becomes $Q'$ in the chart $U_2$ and in every $\pi^{-1}(x), x \in (0,1)$ exhibits a simple singularity in its interior.
\end{enumerate}
Note that these two sets are not open, but the affine structure extends over the two corners.
\end{proof}
Observe that this family provides us both with an affine manifold $B$ -- a general fiber of $\pi$ -- and a polyhedral decomposition $\mathscr{P}$ of $B$, which subdivides $B$ along $l$.  We also require a family of log-structures compatible with the family of affine manifolds.  The line segment $l$ determines a one-dimensional projective toric stack $\P(a,b)$, with the log-structure a section of $\mathcal{O}(\lcm(a,b))$.  The line segment $l$ is the only interior 1-cell so there is no consistency condition to check.  Sections of the bundle $\mathcal{O}(\lcm(a,b))$ are parameterized, up to scale, by $\P^1$ and we pick a family of sections such that the image of the zero set follows the singular locus of the affine structure.  After choosing a piecewise linear $\phi$ on $B$ we can apply the Gross--Siebert algorithm.

Applying the Gross--Siebert algorithm fiberwise, as in Theorem~\ref{thm:families}, and using the local models~\ref{def:icup} to understand the central fiber as in  Proposition~\ref{prop:central}, we obtain families $\pi_i \colon \mathcal{X}_i \rightarrow \Spec \C[\alpha,t]$ for $i = 1,2$.  We now describe these families; as there is no scattering these families are in fact \emph{polynomial} in $t$.  Relating the $\pi_i$ to \cite{Ilten, Overarching}, denote the Ilten family for $Q$,~$Q'$ as $\pi' \colon \mathcal{Y} \rightarrow \P^1$. The construction in \cite{Ilten, Overarching} also defines a family over the affine cone $\C^2$ of $\P^1$, we shall recover this family by gluing together the families $\pi_i$ and contracting the resulting exceptional curve.

\begin{prop}\label{prop:preilten}
There is a family $\pi\colon \mathcal{X} \rightarrow \text{Bl}_0(\C^2)$ from which we obtain each $\pi_i$ as follows.
\begin{enumerate}
\item Cover the base with the standard toric charts $U_1, U_2$.
\item Restricting $\pi|_{U_i}$ to a formal neighbourhood of the exceptional divisor recovers $\pi_i$.
\item The family over the exceptional divisor is trivial, and after restricting to the strict transform of a line in $\C^2$ the family becomes a toric degeneration endowing the restriction of $\pi$ to the exceptional divisor with a family of divisorial log-structures. 
\end{enumerate}
\end{prop}
\begin{remark}
It would be entirely legitimate at this point to embark on a description of this smoothing via the usual local model and inverse limit construction.  For example these must contain the local model:
\[
R_{\tau \tau \mathfrak{u}} \cong S_k[x,y,w^{\pm}]/(xy - (\alpha + w)t)
\]
Indeed, all the families discussed in this section are compactifications of this affine local model.  There are no non-trivial scattering diagrams around any joint of the structure so the family is obtained by taking a colimit over a finite system of algebras.  However, we shall take a different approach, following \cite{Invitation}, which projectivises this construction.  This will greatly reduce the number of rings we need to keep track of and also produce an embedded family with the log structure encoded in the equations defining this family.  We shall prove the equivalence with the original construction in Lemma~\ref{lem:dehomog}.
\end{remark}

Recall that the polygon $P^{\vee} = Q \subset M_\R$ defines a toric variety via $X_{P} = \Proj(\C[C(Q)])$ where $C(Q)$ is the semigroup defined by the integral points of the cone in $M_\R \oplus \R$ with height one slice equal to $Q$.  As the vertices of $Q$ are rational this graded ring need not be generated in degree one.

The prototypical example we shall refer to is the pair of polygons $Q, Q'$ for $\P^2$ and $\P\left(1,1,4\right)$ respectively, they are shown below with the embedding from $\mathcal{O}\left(i\right), i=1,2$ as shown below.

\begin{center}\label{fig:polygons}
\includegraphics*[viewport=149 451 378 668]{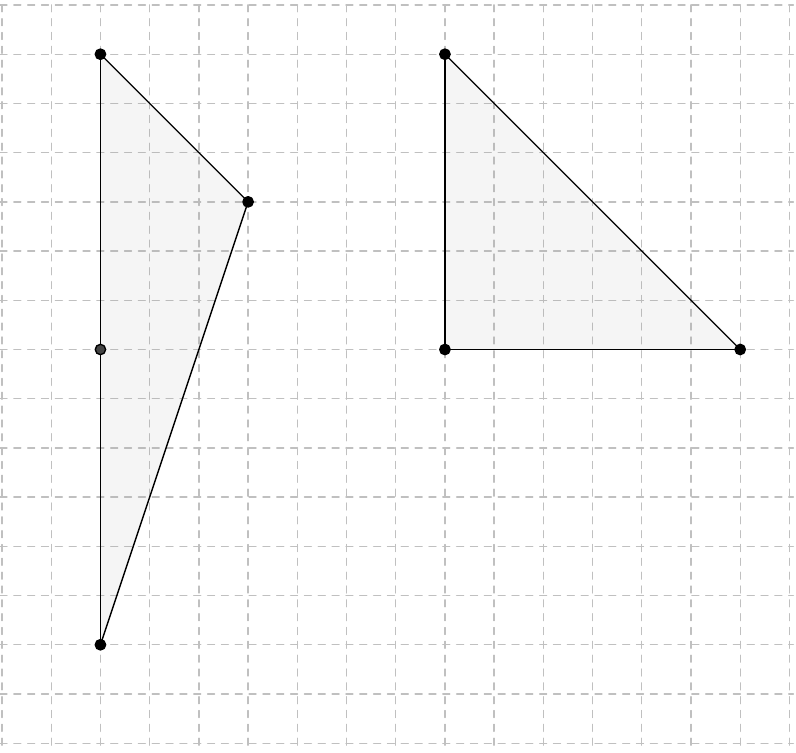}
\end{center}

Take a generating set for $C(Q)$ and refer to a general element of the generating set as $u_i$.  The generating set naturally subdivides into three disjoint sets:
\begin{enumerate}
\item Any generators lying in the cone over $Q_1$ and outside $Q_2$ are denoted $X_i$.
\item Any generators lying in the cone over $Q_2$ and outside $Q_1$ are denoted $Y_j$.
\item Any generators lying over both $Q_1$ and $Q_2$ are denoted $W_k$.  We observe that $(\mathbf{0},1) \in C(Q)$ is always in the generating set.
\end{enumerate}
Indeed we write $C(Q_1), C(Q_2), C(Q_1 \cap Q_2)$ for the three sub-cones respectively.  We shall insist that the union $\{X_i\}\cup\{W_k\}$ generates $C(Q_1)$, $\{Y_j\}\cup\{W_k\}$ generates $C(Q_2)$ and $\{W_k\}$ generate $C(Q_1 \cap Q_2)$. We denote the height of a generator $u_i$ as $\kappa(u_i)$.

\begin{remark}
In the example above we can take a generating set with four elements, which we shall call $\left\{s_0,s_1,s_2,u\right\}$ with heights $1,1,1,2$ respectively.  Thus we see $\P^2$ embedded as $s_1s_2 = u$ and $\P\left(1,1,4\right)$ embedded as $s_1s_2 = s_0^2$ in $\P(1,1,1,2)$.
\end{remark}

Recalling that the affine manifold is equipped with a piecewise-linear function $\phi$, we assume this has slope zero on $Q_2$ and slope $k$ on $Q_1$, i.e.~$\phi(X_i)$ is $k\left\langle n_0,\tilde{m}_i \right\rangle$ where $n_0$ is the primitive vector in $N$ annihilating the tangent space to $l$, and $\tilde{m}_i$ is the rational point of $Q$ defined by the exponent $m_i$ of $X_i$.  We shall assume $k$ is chosen such that $\phi$ is integral on each generator. We can now write out the $\Proj$ of this algebra explicitly: we can construct an ambient weighted projective space $\P(\vec{a})$, where $\vec{a} \in \Z^{N}_{>0}$ and $N$ is the size of the generating set, given by $\vec{a} = \sum_i \kappa(u_i) e_i$, the vector of heights.

The toric variety is then cut out in this space by the binomial equations given by the relations between these generators. We call the ideal generated $I_Q$.  The toric degeneration corresponding to $\mathscr{P}$ is given by the following ideal, denoted $I_P(t)$:

\begin{definition}
For each binomial relation $M_1 - M_2 \in I_P$ such that $d = \ord_l(M_1) - \ord_l(M_2) \geq 0$ define a new binomial relation $M_1 - t^dM_2$. Take $I_P(t)$ to be the ideal generated by these new relations.
\end{definition}

\begin{remark}
If $F \in I_P$ is an element of $\C[\{X_i\}\cup\{W_k\}]$, then $\ord_l(M_1) - \ord_l(M_2) = 0$ and the binomial relation remains unchanged in $I_P(t)$.  The same is true of those relations in $\C\left[\left\{Y_j\right\}\cup\left\{W_k\right\}\right]$
\end{remark}

Note this has recovered the Mumford degeneration for the pair $\left(Q,\mathscr{P}\right)$.  We have thus completed the first step, this family will be the family over the strict transform of a line through the origin in $\C^2$.

\begin{remark}
One can apply exactly the same procedure to $Q'$ and obtain a toric degeneration of the second toric variety, the family over the fiber at $\infty$.  In fact one may take exactly the same generating set, and get a different set of binomial relations.  As in Section~\ref{sec:smoothing} we now describe a family `interpolating' between them.
\end{remark}

To construct such a family first consider that in the construction in Section~\ref{sec:smoothing} we used a variable that corresponded to a primitive vector along the monodromy invariant direction.  In this construction we find such a variable by looking at the part of $C(Q)^{\text{gp}}$ generated by the exponents of the variables $W_k$.  This is a rank 2 free abelian subgroup of $C(Q)^{\text{gp}}$, that contains $(\mathbf{0},1)$.  There is another canonical monomial $\mathcal{W}$, determined up to sign by requiring it to lie at height zero and lie in the monodromy invariant direction.  In $\C\left[C(Q)^{\text{gp}}\right]$ this has the form $\mathcal{W} = \frac{\prod_k{W^{\alpha_k}_k}}{\prod_l{W^{\beta_l}_l}}$.  Note there may be many choices for the representation of $\mathcal{W}$ via the relations between the $W_k$.

\begin{remark}
In the example of $\P^2 \subset \P\left(1,1,1,2\right)$ we may take $\mathcal{W} = u/s_0^2$.
\end{remark}

The interpolating family is then given by replacing elements in $I_Q(t)$ analogously to the procedure in Section~\ref{sec:local_models}:
\begin{definition}
The ideal $I_Q(t,\alpha)$ is the ideal generated by relations defined in Definition~\ref{def:icup}, where we replace $C_i$ by $Q_i$ and $f$ by $(1+\alpha \mathcal{W})$.
\end{definition}

In the example we have been considering, for $\P^2 \subset \P\left(1,1,1,2\right)$, we replace the relation $s_1s_2 = u$ with $s_1s_2 = ut(1+\alpha s_0^2/u)$ i.e. with $s_1s_2 = t(u + \alpha s_0^2)$.  Observe that the fibers of this family are isomorphic to $\P^2$.  The other family, that deforming $\P\left(1,1,4\right)$, is given by $s_1s_2 = t(s_0^2 + \alpha u)$. This gives a smoothing of $\P\left(1,1,4\right)$ to $\P^2$.

To complete a proof of Proposition~\ref{prop:preilten} we glue this pair of families in the obvious fashion. Define $\mathcal{X} \rightarrow \text{Bl}_0(\C^2) =: E$ by taking $\mathcal{X} \hookrightarrow \P(\vec{a}) \times E$.  Giving $E$ homogenous co-ordinates, $\alpha, \beta$ of weight one and $t$ the weight $-1$ co-ordinate, elements of $I_{P}(t,\alpha)$ may be homogenized to obtain: $M_1 = t^d(\beta+\alpha \mathcal{W})^dM_2$ homogenous of weight zero.  These generate a homogeneous ideal, the equations of which define $\mathcal{X}$.

Given the family produced by Proposition~\ref{prop:preilten} we can establish a family over $\C^2$ by contracting the exceptional curve, so that $\alpha$ and $\beta$ become the coordinates on the plane and the new family is defined by equations $M_1 = (\beta+\alpha \mathcal{W})^dM_2$. Thus we have established Theorem~\ref{thm:Ilten_in_introduction}.

In the running example the homogeneous equation is:
\[
\{s_1s_2 = (\beta s_0^2 + \alpha u)\} \subset \P(1,1,1,2) \times \P^2_{(t:\alpha:\beta)}
\]
\begin{lem}\label{lem:dehomog}
Restricting to the ideal of $\C[\alpha]\llbracket t \rrbracket$ generated by $(\alpha - \alpha_0, t^{k+1})$ for fixed $\alpha_0 \neq 0$ denote the restriction of $\cX$ by $\mathcal{X}_{\alpha_0, k}$, this scheme is isomorphic to the scheme obtained in Sections~\ref{sec:scattering},~\ref{sec:gluing} from $(B,\mathscr{P})$ with log-structure fixed by the parameter $\alpha$.
\end{lem}

\begin{proof}
Considering this $(B,\mathscr{P})$, there is no scattering, so we have $\mathscr{S}^r = \varnothing$, and the set of slabs $\mathscr{S}^s = \{l\}$.  The category $\underline{\text{Glue}}(\mathscr{S},k)$ consists of objects $(\omega, \tau, \mathfrak{u})$ where:
\begin{enumerate}
\item $\omega$ is an end-point of $l$, $\tau = l$ and $\mathfrak{u}$ is either of the two maximal cells of $\mathscr{P}$.
\item In any other case the chamber is fixed by the choice of $\omega, \tau$.  In particular $\tau$ is a boundary edge of $B$ and contained in precisely one two-cell of $\mathscr{P}$.
\end{enumerate}
Firstly $R^k_\omega$ is recovered by localizing $\mathcal{X}_{\alpha,k}$ with respect to the variable $W_k$ corresponding to the vertex $\omega$ in $C(Q)$.  This is immediate from the usual $\Proj$ construction and performing this localisation we recover $R^k_\cup$ for this vertex, by construction.  Indeed the same argument applies for any vertex of $Q$.  The final check is that the gluing of these rings according to Section~\ref{sec:gluing} coincides with that of $\Proj$.
\end{proof}

\begin{cor}\label{cor:localilten}
The family given by Theorem~\ref{thm:Ilten_in_introduction} is $\Q$-Gorenstein.
\end{cor}
\begin{proof}
We can cover the family by neighbourhoods around each boundary vertex.  By Lemma~\ref{lem:dehomog} each of these is equal to the local model described in Section~\ref{sec:local_models} and is therefore $\Q$-Gorenstein.
\end{proof}
We remark the analogous families in both \cite{Ilten} and \cite{Overarching} are independently known to be $\Q$-Gorenstein, making this an expected outcome.

\section{Examples}

\subsection{Rigid del~Pezzo surfaces}

Given a Fano polygon $Q\subset N_{\R}$ there may be no way of exchanging any of its corners with singularites in the interior of the affine manifold at all.  In the language of \cite{SingCon} this is the statement that all the singularities of the corresponding toric variety $X_Q$ are \emph{residual singularities}, and so  $X_Q$ is $\Q$-Gorenstein rigid (see \cite{Overarching}).  The standard example of this phenomenon is $\P(3,5,11)$, though it may be thought of as `generic' behaviour. 

\subsection{A single smoothing direction}
\label{sec:A1-type}
Consider the hypersurface:
\[
X_6 \subset \P(1,3,3,1)
\]
This exhibits a toric degeneration in this ambient space to a toric variety with fan:
\begin{center}
\includegraphics*[viewport=156 563 344 635]{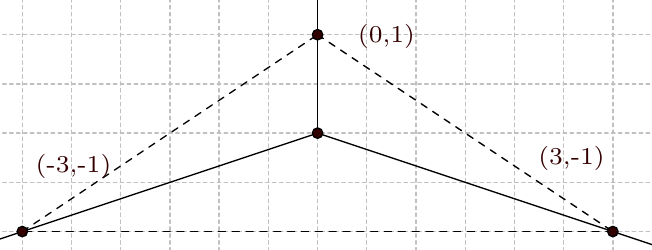}
\end{center}
The fan exhibits 2 residual singularities which persist after the smoothing and an $A_5$ singularity, $\frac{1}{6}(1,5)$ which is a $T$-singularity.  Constructing the dual polygon one observes that the one-parameter family of affine manifolds obtained by smoothing all possible corners has a general fiber $B$ with all six singularities ranged along a single edge. Therefore there is no scattering diagram to construct so one can construct a family (the multi-parameter analogue of the family appearing in Section~\ref{sec:ilten}) for which all the mutation equivalent toric varieties are special fibers.

To write down the family constructed in Section~\ref{sec:ilten} for this polygon we consider the dual polygon $Q^\vee$:
\begin{center}
\includegraphics*[viewport=174 540 320 655]{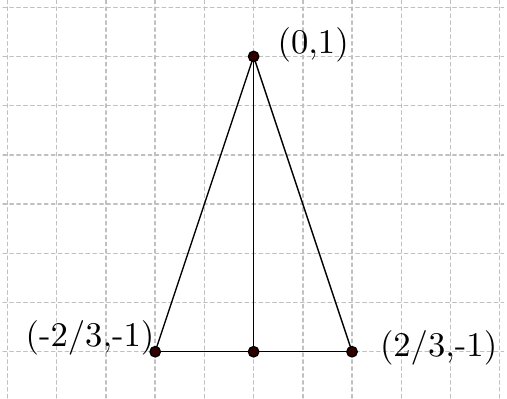}
\end{center}
Now form the monoid of integral points of the cone for which $Q^\vee$ is the height one slice.  However, note that the polygon is that obtained from the polarisation $\mathcal{O}(2)$; using the more economical polarisation $\mathcal{O}(1)$ (embedding $Q^\vee$ at height 2) the associated relation is a binomial in $\P(1,1,3,3)$.  Indeed the vertices of the polygon at height one are now $(0,1), (0,0), (-1/3,0),(1/3,0)$ after a translation, naming the corresponding variables $X_0,X_1,Y,Z$ respectively gives: $YZ = X^6_1$.  Applying the method of Section~\ref{sec:ilten}, we find the Ilten family:
\[
\{ YZ  = (\alpha X^6_1 + \beta X^5_1X_0) \}
\]
Of course we can consider a general homogenous degree six polynomial in $X_0,X_1$ and so find a family over $\P^5$ which has $6$ toric zero strata, each element of which corresponds to a particular toric variety.  There is redundancy in this description, since for example $YZ = X^6_0$ manifestly gives the same variety as $YZ = X^6_1$.

\subsection{The cubic surface}

In this example we place Example 4.4 of \cite{Invitation} in this context.  The toric cubic surface $\{X_0X_1X_2 = X_3^3\} \subset \P^3$ exhibits $3\times A_2$ singularities which may all be smoothed.  However this situation is much more chaotic than the previous examples -- the mutation graph is necessarily infinite and we cannot expect to capture all degenerations in a single algebraic family.  However following \cite{Invitation} we may ask an easier question; rather than smoothing the corners completely we can simply introduce three type 1 singularities.  This should produce a family of cubic surfaces which all exhibit at least ordinary double points.  In \cite{Invitation} this scattering diagram is explicitly computed, in particular it is shown to be finite, producing a toric degeneration embedded in $\P^3$.

Having produced the scattering diagram one can construct a toric degeneration as explained above.  The equation from \cite{Invitation} is:
\[
\{ XYZ = t((1+t)U^3 + (X+Y+Z)U^2 )\} \subset \P^3\times \C_t
\]
To recover the family partially smoothing these $A_2$ singularities we simply repeat the derivation of this, but place general coefficents in the sections defining the log-structure.  We know from Section~\ref{sec:local_models} that this will give the correct family as these sections degenerate.

This calculation gives a family over $\C^3_{\alpha,\beta,\gamma}$:
\[
\{XYZ = t((1+\alpha\beta\gamma t)U^3 + (\alpha X + \beta Y + \gamma Z)U^2)\}
\]
For completeness we also compute an Ilten family for the cubic surface:\newline
Subdividing using the $x$-axis, we have zero strata:
\[
(1,0),(0,1),(-1,-1),(0,0),(-1/2,0)
\]
Naming the corresponding variables $X,Y,Z,U,W$ respectively we obtain the toric degeneration:
\[
\{XYZ = tU^3, YZ=tW\} \subset \P(1,1,1,1,2)
\]
Performing the construction of Section~\ref{sec:ilten} we obtain the family:
\[
\{ XYZ = tU^2(\alpha U + \beta X), YZ=t(\alpha W + \beta U^2)\} \subset \P(1,1,1,1,2)
\]





\subsection{Polygons of finite mutation type}

We say that a Fano polygon $Q$ has \emph{finite mutation class} if it is mutation equivalent to only finitely many polygons.  In this case, the scattering diagram considered above is \emph{finite}, and so the Gross--Siebert reconstruction algorithm terminates after finitely many steps.  Thus, in this case, we can construct the family $\mathcal{X}_Q$ explictly. The families $\mathcal{X}_{Q'}$, where $Q'$ is mutation-equivalent to $Q$, patch together to form a single family $\mathcal{X}$ that contains, as special fibers, all toric degenerations of its generic fiber.

In \cite{Prince} we shall classify Fano polygons with finite mutation class. In \cite{Minimality} notions of quivers and cluster algebras associated to polygons were introduced. Using the classification of cluster algebras of finite type and finite mutation type, we shall classify in \cite{Prince} those Fano polygons which admit finitely many polygons in their mutation equivalence class.  These may be divided into classes of type $A_1^k$, for $k \in \Z_{\geq 0}$,~$A_2$,~$A_3$ and $D_4$. The $A_1^k$ case equates to the examples covered in section~\ref{sec:A1-type}, but for any type the scattering diagram one obtains at the origin is finite, and so the output of the Gross--Siebert algorithm may be explicity computed in precisely these cases.

\section{Conclusion}

An intuitive picture begins to emerge: If we fix a del~Pezzo surface $X$ which is a smoothing of a toric variety $X_P$ we have various \emph{mutation equivalent} toric varieties, namely those associated to the polygons obtained by mutating $P$.  Rather than directly analysing the deformation theory of these varieties we studied the moduli space of log-structures after taking a toric degeneration of $X$.  This produced a `tropical analogue' of the deformation theory, in which one mimics the $\Q$-Gorenstein deformations of $X_P$ by introducing singularities into the affine manifold $P$.  As well as recovering the entire theory of combinatorial mutations we have shown how to recover, order by order, an algebraic family with general fiber $X$ via the Gross--Siebert algorithm.

Moving singularities defines a `moduli problem' of its own, a topological orbifold (due to automorphisms of the polygons) which carries an affine structure, first mentioned in \cite{KoSo}. There is also a stratification of this space: The zero strata being the polygons themselves, one strata the tropical Ilten families and so on.  To relate this space to the study of $\Q$-Gorenstein degenerations one must understand how to lift these families to algebraic ones.  From this perspective we have described this lift for the 1-skeleton of this space in this article.

Finally, we recall that the techniques used in this article are motivated by results in mirror symmetry.  As mentioned in Remark~\ref{rem:Auroux} the geometric interpretation of the scattering process is that it records \emph{instanton corrections}, which in this context Maslov index zero holomorphic discs.  In fact in the cases where the scattering diagram is finite one may hope to gain a completely geometric understanding of the situation.  For example in \cite{Aur1} the case of a single singularity treated: the affine base of $(\C^2,C)$ for a conic $C$ is computed from a torus fibration and the Maslov index zero discs are computed.  Compactifying this model in different ways would recover the Ilten families once again.

Whilst we have attempted no mirror symmetry calculations in this article, the shape of such results is already visible from \cite{Aur1} and \cite{TropicalLG}.  In particular taking the \emph{Legendre dual} one would recover the various Laurent polynomials from counts of \emph{broken lines}.  Taking the affine manifold obtained as a general fiber of the tropical Ilten family, the dual base manifold also has a single wall and a suitable broken line count shows that crossing this wall induces precisely the desired mutation of the Laurent polynomial. More concisely: `\emph{The Ilten family is mirror to the mutation}'.  Smoothing more corners one must consider affine manifolds of the form considered in \cite{TropicalLG}; here the scattering process is more complicated but one may expect to see a wall and chamber decomposition with the Laurent polynomials lying on each chamber related by mutations.  We shall return to this in a future work.

\bibliographystyle{plain}
\bibliography{smoothing_surfaces}

\end{document}